\documentclass[10pt]{article}
\usepackage{amsfonts,amssymb,amsmath}

\usepackage[pagebackref]{hyperref}

\usepackage{xcolor}

\numberwithin{equation}{section}
\def\N_0{\mathbb{N}_0}
\def\balpha{\boldsymbol{\alpha}}

\def\aff{\textrm{Aff\/}}

\def\X{\mathcal{X}}

\def\K{\mathbb{F}_q}

\newtheorem{theorem}{Theorem}[section]
\newtheorem{lemma}[theorem]{Lemma}
\newtheorem{proposition}[theorem]{Proposition}

\newtheorem{defin}[theorem]{Definition}

\newenvironment{proof}{\noindent \textbf{Proof: }}{\hfill
$\Box$  \vspace{1ex}}
\newenvironment{definition}{\begin{defin}\em}{\end{defin}}
\newtheorem{defins}[theorem]{Definitions}
\newenvironment{definitions}{\begin{defins}\em}{\end{defins}}
\newtheorem{exs}[theorem]{Examples}

\newtheorem{ex}[theorem]{Example}

\newtheorem{rem}[theorem]{Remark}

\newtheorem{rems}[theorem]{Remarks}

\newtheorem{corollary}[theorem]{Corollary}

\newcommand{\negrito}[1]{\mbox{{\boldmath$#1$}}}

\begin{document}



\begin{center}
{\Large\textbf{An extension of Delsarte, Goethals and Mac Williams theorem on 
minimal weight codewords to 
a class of Reed-Muller type codes }}
\end{center}
\vspace{3ex}

\noindent\begin{center} 
\textsc{C\'{\i}cero  Carvalho and Victor G.L. Neumann}\footnote{Both authors
were partially supported by  grants from CNPq and FAPEMIG.} \\
{\em\small Faculdade de Matem\'atica \\ Universidade Federal de Uberl\^andia \\ 
Av.\ J.\ N.\ \'Avila 2121, 38.408-902 - Uberl\^andia - MG, Brazil 
\\{\small  cicero@ufu.br \ \ \ \ \ \ \ \ \  victor.neumann@ufu.br}} \end{center}
\vspace{3ex}


\noindent
{\footnotesize \textbf{Abstract.} 
 In 1970 Delsarte, Goethals and Mac Williams published a seminal paper on 
 generalized Reed-Muller codes where, among many important results, they 
 proved that the minimal weight codewords of these codes are obtained through 
 the evaluation of certain polynomials which are a specific product of linear 
 factors, which they describe. In the present paper we extend this result to a 
 class of Reed-Muller type codes defined on a product of (possibly distinct) 
 finite fields of the same characteristic. The paper also brings an expository 
 section on the study of  the 
 structure of low weight codewords, not only for affine Reed-Muller type 
 codes, but also for the projective ones.}
\vspace{2ex}

\section{Introduction with a hystorical survey}\label{intro}

Let $\K$ a field with $q$ elements, let $K_1,\ldots, K_n$ be 
a collection of non-empty subsets of $\K$, and let 
\begin{eqnarray*}
\mathcal{X}:= K_1\times\cdots\times K_n:=\left\{(\alpha_1:\cdots :\alpha_n) \vert\, \alpha_i\in
K_i
\mbox{ for all } i\right\}\subset \K^n. 
\end{eqnarray*}

Let $d_i := | K_i |$ for $i = 1,\ldots, n$, so 
clearly $| \mathcal{X}| = \prod_{i = 1}^n d_i =: m$, and let 
$\mathcal{X} = \{\balpha_1, \ldots, \balpha_m\}$. It is not difficult to check that the
ideal of polynomials in $\K[X_1, \ldots, X_n]$ which vanish on $\mathcal{X}$ is
$I_\mathcal{X} = (\prod_{\alpha_1 \in K_1} (X_1 - \alpha_1),\ldots ,  \prod_{\alpha_n \in K_n} (X_n - \alpha_n))$ (see e.g.
\cite[Lemma 2.3]{lopez-villa} or \cite[Lemma 3.11]{car2}). From this we get that
the evaluation morphism $\Psi: \K[X_1, \ldots, X_n]/I_{\mathcal{X}} \rightarrow
\K^m$  given by $P + I_{\mathcal{X}} \mapsto ( P(\balpha_1), \ldots, 
P(\balpha_m) )$ is 
well-defined and
injective. Actually, this is an isomorphism of $\K$-vector spaces because for
each $i \in \{1, \ldots, m\}$ there exists a  polynomial $P_i$ such that
$P_i(\balpha_j)$ is equal to $1$, if $j = i$, or $0$, if $j \neq i$, so that $\Psi$ is
also surjective.

\begin{definition}  \label{def1}
Let $d$ be a nonnegative integer. 
The {\em affine cartesian code} (of order $d$) $\mathcal{C}_\mathcal{X}(d)$
defined over the sets $K_1, \ldots, K_n$ is the image, by $\Psi$, of the set of
the classes of all polynomials of degree up to $d$, together with the class of
the zero polynomial.
\end{definition}  
 
%
%

These codes appeared
independently  in \cite{lopez-villa} and \cite{Geil} (in \cite{Geil}  in a 
generalized form). In the special case
where $K_1 = \cdots = K_n = \K$ we have the well-known generalized Reed-Muller
code of order $d$. In \cite{lopez-villa} the authors prove that we may ignore,
in the cartesian product, sets with just one element and moreover may always
assume that $2 \leq d_1 \leq \cdots \leq d_n$. They also determine the dimension
and the minimum distance of these codes.

For the generalized Reed-Muller codes, the classes of the polynomials whose
image are the codewords of minimum weight were first described explicity by  
Delsarte,
Goethals and Mac Williams in 1970. This result started a series of 
investigations of the structure of codewords of all weights, not only in 
generalized Reed-Muller codes, but also in related Reed-Muller type codes.  
In the present paper we extend the result of Delsarte, Goethals and Mac Williams 
to  affine cartesian codes, in the case where $K_i$ is a field,  for all $i =1,
\ldots, n$ and $K_1 \subset K_2 \subset \cdots \subset K_n \subset \K$, but 
before we describe the contents of the next sections of this work, we would 
like to present a survey of results that pursued the investigation started by 
Delsarte, Goethals and Mac Williams. 

Reed-Muller codes are binary codes defined by Muller (\cite{muller}) and were 
given a decoding algorithm by Reed (\cite{reed}), in 1954. In 1968 Kasami, Lin 
and Peterson (\cite{klp}) introduced what they called the generalized 
Reed-Muller codes, defined over a finite field $\mathbb{F}_q$ with $q$ 
elements, which coincided with Reed-Muller codes when $q = 2$. Their idea was 
to 
consider the $\mathbb{F}_q$-vector space $\mathbb{F}_q[X_1, \ldots,  X_n]_{\leq 
d}$ 
of all 
polynomials in 
$\mathbb{F}_q[X_1, \ldots,  X_n]$ of  degree less or equal than $d$, together 
with the zero polynomial, for some 
positive integer $d$, and define the generalized Reed-Muler code of order $d$ 
as 
\[
GRM_q(d,n) = \{ (f(\balpha_1), \ldots, f(\balpha_{q^n}) \in 
\mathbb{F}_{q}^{q^n} \, | \, f 
\in 
\mathbb{F}_q[X_1, \ldots,  X_n]_{\leq 
d} \}
\]
where $\balpha_1, \ldots, \balpha_{q^n}$ are the points of the affine space 
$\mathbb{A}^n(\mathbb{F}_q)$. Equivalently, using the fact that $I = (X_1^q - 
X_1, \ldots, X_n^q - X_n)$ is the ideal of 
polynomials whose zero set is $\mathbb{A}^n(\mathbb{F}_q)$, we have that
$GRM_q(d,n)$ is the image of the linear transformation 
$\Psi: \K[X_1, \ldots, X_n]/I \rightarrow
\K^{q^n}$  given by $P + I_{\mathcal{X}} \mapsto ( P(\balpha_1), \ldots, 
P(\balpha_{q^n}) )$.
Kasami et al.\ proved that if $d \geq n (q - 1)$ then  we have 
$GRM_q(d,n) = \mathbb{F}_{q}^{q^n}$ hence the minimum distance 
 $\delta_{GRM_q(d,n)}$ 
of $GRM_q(d,n)$ is 1. For $1 \leq d < n(q - 1)$ write $d = k (q - 1) + \ell$ 
with 
$0 < \ell \leq q - 1$, then  $\delta_{GRM_q(d,n)} = (q - \ell) q^{n - k - 1}$ 
(see 
\cite[Thm. 5]{klp}). McEliece, studying quadratic forms defined over 
$\mathbb{F}_q$ (see \cite{mceliece}) described the so-called weight enumerator 
polynomial for $GRM_q(2,n)$, i.e. described all possible 
weights for the codewords in  $GRM_q(2,n)$, together with the number of 
codewords 
of each 
weight, and also gave canonical forms for the polynomials whose classes  
produced codewords of all  weights.

In 1970 Delsarte, Goethals and Mac Williams published a 
40 pages seminal paper which started the systematic study of the generalized 
Reed-Muller codes  
and 
other codes related to them. Among the many important results in the paper, 
there is a 
description of the polynomials whose evaluation yields the codewords with 
minimum 
distance. To state their result, we recall that the affine group of 
automorphisms of $\mathbb{F}_q[X_1, \ldots, X_n]$ is the one given by 
transformations of the type $\boldsymbol{X}^t \mapsto A \boldsymbol{X}^t + 
\boldsymbol{\beta}$, where $\boldsymbol{X} = (X_1, \ldots, X_n)$,  $A$ is a $n 
\times 
n$ invertible matrix with entries 
in $\mathbb{F}_q$ and $\boldsymbol{\beta} \in \mathbb{F}_q^n$.  

\begin{theorem}\cite[Theorem 2.6.3]{dgm}
The minimal weight codewords of $GRM_q(d,n)$ come from the evaluation  of 
$\Psi$ in classes $f + I$ of polynomials $f$ which, after a suitable action of 
an affine automorphism of  $\mathbb{F}_q[X_1, \ldots, X_n]$, may be written as
\[
f = \alpha  \prod_{i = 1}^k (X_i^{q - 1} - 1) \prod_{i = 1}^\ell (X_{k + 1} - 
\beta_j)
\]
where $d = k (q - 1) + \ell$ with 
$0 < \ell \leq q - 1$, $\alpha \in \mathbb{F}_q^*$ and $\beta_1, \ldots, 
\beta_\ell$ are distinct elements of $\mathbb{F}_q$ (in the case $k = 0$ we 
take the first product to be 1).
\end{theorem}     

Since GRM codes arise from the evaluation of polynomials in points of an 
affine space, there is also an algebraic geometry interpretation for the 
codewords. In fact, the above theorem shows that the zeros of a minimal weight 
codeword lie on a special type of hyperplane arrangement. More explicitly, we 
have the following alternative statement (taken from \cite{ballet}) for the 
above result.

\begin{theorem}
Let $V$ be an algebraic 
hypersurface in  $\mathbb{A}^n(\mathbb{F}_q)$, of degree at most $d$, with 
$1 \leq d < n(q - 1)$, which is not the whole $\mathbb{A}^n(\mathbb{F}_q)$. 
Then $V$ has the maximal possible number of	 zeros if and only if 
\[
V = \left( \bigcup_{i = 1}^k \left( \bigcup_{s = 1}^{q - 1} V_{i,s} \right) 
\right) \cup
\left( 
\bigcup_{j = 1}^\ell W_j \right)
\]
where $d = k (q - 1) + \ell$ with 
$0 \leq \ell < q - 1$, the $V_{i,s}$ and  $W_j$ are $d$ distinct hyperplanes 
defined on $\mathbb{F}_q$ such
that for each fixed $i$ the $V_{i,s}$ are $q-1$  parallel hyperplanes, the
$W_j$ are $\ell$ parallel hyperplanes
and the $k + 1$ distinct linear forms directing these hyperplanes are linearly 
independent.
\end{theorem} 

This result was the start of the search for the higher Hamming weights together 
with the description (algebraic and geometric) of the codewords having these 
weights, not only for GRMs but in general for all Reed-Muller type codes, like 
the ones studied in this paper, for the GRMs alone the search is still ongoing.

In 1974 Daniel Erickson, a student of  McEliece and Dilworth, 
devoted his Ph.D. thesis to the determination of the 
second lowest Hamming weight, also called next-to-minimal weight, of 
$GRM_q(d,n)$ (see \cite{erickson}). He succeeded in determining the values of 
the second weight for 
many values of $d$ in the relevant range $1 \leq d < n(q - 1)$. For the values 
that he was not able to determine, following a suggestion by M.\ Hall, he 
generalized some of the results of Bruen on blocking sets, which had appeared 
in 
\cite{bruen1}, and made a conjecture relating the expected value for the 
missing weights to the cardinality of certain blocking sets in the affine plane 
$\mathbb{A}^2(\mathbb{F}_q)$. Also, instead of 
working with the classes of polynomials in  
$\K[X_1, \ldots, X_n]/I$ he worked with a fixed set of representatives  
called ``reduced polynomials'' which he noted that were in a one-to-one 
correspondence with the functions from $\mathbb{F}_q^n$ to $\mathbb{F}_q$. This 
had an influence on  the paper \cite{leducq} and also the present text, as we 
will comment later. Unfortunately Erikson's results were not published, and 
the quest for the next-to-minimal weights of GRM codes went on for many years 
without his contributions.  

In 1976 Kasami, Tokura and Azumi (see \cite{kta}) determined all the weights of 
$GRM_2(d,n)$ (i.e. 
Reed-Muller codes) which are less 
than  $\dfrac{5}{2}\delta_{GRM(d,2)}$. They also determined canonical forms for the 
representatives of the classes whose evaluation produces codewords of these 
weights, together with the number of such words. In particular, the second 
weight of Reed-Muller codes was determined. After this paper, there was not 
much work done on the problem of determining the higher Hamming weights of 
$GRM_q(d,n)$ during two 
decades. Then, in 1996 Cherdieu and Rolland (see \cite{cherdieu-rolland}) 
determined the second weight of 
$GRM_q(d,n)$ for $d$ in the range $1 
\leq d < q - 1$, provided that $q$ is large enough. They also proved that in 
this case the zeros of codewords having next-to-minimal weight form an specific 
type of hyperplane arrangement which they describe. 
In the following year a  work by Sboui (see \cite{sboui}) proved that 
the result by Cherdieu and Rolland holds when $d \leq q/2$.

In 2008 Geil (see \cite{Geil2} and \cite{GeilErratum}) determined the second 
weight of $GRM_q(d,n)$ 
for 
$2 \leq d \leq q - 1$ and $2 \leq n$. Also, for $d$ in the range $(n - 1)(q - 
1) < d < 
n (q - 1)$, he 
determined the first $d + 1 - (n - 1)(q - 1)$ weights of $GRM_q(d,n)$.
His results completely determine the next-to-minimal weight of $GRM_q(d,2)$, 
since in this case the relevant range for $d$ is $1 \leq d < 2 q$.
Geil's theorems  were obtained using results from Gr\"obner basis theory. In 
2010 Rolland made a more detailed analysis of the weights also using Gr\"obner 
basis theory results, and determined almost all 
next-to-minimal 
weights of $GRM_q(d,n)$ (see \cite{rolland}). In fact, he succeeded in finding 
the next-to-minimal weights  for all values of $d$, in the range $q \leq d < n 
(q - 1)$, that 
can not be written in the form $d = k (q - 1) + 1$. 
Finally, also in 2010, A.\ 
Bruen had his attention directed to Erickson's thesis, and in a note (see 
\cite{bruen2}) observed that Erickson's conjecture was an easy consequence of 
results that he, Bruen, had proved in 1992 and 2006 (see \cite{bruen3} and 
\cite{bruen4}). This finally completed the determination of the next-to-minimal 
weights 
$\delta^{(2)}_{GRM_q(d,n)}$
of $GRM_q(d,n)$, and now we know that for $1 \leq d < n(q - 1)$, writing $d = k 
(q - 1) + \ell$ 
with 
$0 \leq \ell < q - 1$, then  $\delta_{GRM_q(d,n)} = (q - \ell) q^{n - k - 1}$
and $\delta^{(2)}_{GRM_q(d,n)} = \delta_{GRM_q(d,n)} + c q^{n - k - 2}$, where
\[
c = \left\{ \begin{array}{lcl}  
q &     \textrm{ if } & k = n - 1; \\
\ell - 1 & \textrm{ if } & k < n - 1 \textrm{ and } 1 < \ell \leq (q + 
1)/2;            \\
      & \textrm{ or } & k < n - 1 \textrm{ and } \ell = q - 1  \neq 1;         
      \\
q      & \textrm{ if } & k = 0 \textrm{ and } \ell = 1;      \\
q - 1  & \textrm{ if } & q < 4, 0 < k < n - 2, \textrm{ and } \ell = 
1;           
\\
q - 1  & \textrm{ if } & q = 3, 0 < k = n - 2 \textrm{ and } \ell = 1;          
\\
q      & \textrm{ if } & q = 2,    k = n - 2 \textrm{ and } \ell = 1;           
\\
q      & \textrm{ if } & q \geq 4, 0 <  k \leq n - 2 \textrm{ and } \ell = 
1;           \\
\ell - 1 & \textrm{ if } & q \geq 4,  k \leq n - 2 \textrm{ and } (q + 1)/2 < 
\ell.
\end{array}  \right.
\]

In 2012 the 1970's theorem of Delsarte, Goethals and Mac Williams was the 
subject 
of a paper by Leducq (see \cite{leducq}). In their paper, Delsarte et al.\ 
prove the theorem on the minimum distance in an Appendix entitled ``Proof  of  
Theorem 2.6.3.'', which opens with the 
sentence: ``The  authors  hasten  to  point  out  that  it  would  be  very  
desirable  to  find  a 
more  sophisticated  and  shorter  proof.''  Leducq indeed provides a shorter 
and less technical proof, treating the codewords as functions from 
$\mathbb{F}_q^n$ to $\mathbb{F}_q$ and using results from affine geometry. Some 
of these results appear in the appendix of Delsarte et al.\ paper, and were 
also used by Erickson in his work. In the following year, Leducq (see 
\cite{leducq2}) completed the 
work of previous researchers, with Sboui, Cherdieu, Rolland and Ballet among 
them, and 
proved that the next-to-minimal weights are only attained by codewords whose 
set of zeros form certain hyperplane arrangements.  In the same year
 Carvalho (see \cite{carvalho}) extended Geils's results of 2008 to 
affine cartesian codes, also determining a
series of higher Hamming weights for these codes. 

In 2014 a paper by Ballet and Rolland (see \cite{ballet}) presented bounds on 
the third and fourth Hamming weights of $GRM_q(d,n)$ for certain ranges of $d$.	
In the following year Leducq (see \cite{leducq3}), pursuing and developing 
ideas from Erickson's thesis, determined the third weight and characterized the 
third weight words of $GRM_q(d,n)$ for some values of $d$. 
In 2017 Carvalho and Neumann (see \cite{car-neu2017}) extended many of the 
results of Rolland, in \cite{rolland}, to affine cartesian codes. They found 
the second weight of these codes for all values of $d$ which can not be written 
as $d = \sum_{i = 1}^k (d_i - 1) + 1$, and they also prove that the weights 
corresponding to such values of $d$ are attained by codewords whose set of 
zeros are hyperplane arrangements (yet they don't prove that every word 
attaining those next-to-minimal weights comes from hyperplane arrangements).

There is a ``projective version'' of the generalized Reed-Muller codes whose 
parameters have been studied like those of $GRM_q(d,n)$ and to which they are 
related. This version was introduced by Lachaud in 1986 (see \cite{lachaud2}), 
but one can find some examples of it already in \cite{vla-manin}. 

Let $\boldsymbol{\gamma}_1,\ldots, \boldsymbol{\gamma}_N$ be the points of 
$\mathbb{P}^n(\mathbb{F}_q)$, where $N = q^n + 
\cdots + q + 1$. From e.g.\ \cite{idealapp} or 
\cite{mercier-rolland}  we get that the homogeneous ideal $J_q \subset 
\mathbb{F}_q[X_0, \ldots 
, X_n]$ of the polynomials which vanish in all points of 
$\mathbb{P}^n(\mathbb{F}_q)$ is 
generated by $\{ X_j^q X_i - X_i^q X_j\, | 
\, 0 \leq i < j \leq n\}$. We denote by $\mathbb{F}_q[X_0, \ldots , X_n]_d$
 (respectively, $(J_q)_d$) the $\mathbb{F}_q$-vector subspace formed by the 
homogeneous polynomials of degree $d$ (together with the zero polynomial) in  
$\mathbb{F}_q[X_0, \ldots , X_n]$ (respectively, $J_q$).

\begin{definition}
Let $d$ be a positive integer and 
let $\Theta: \mathbb{F}_q[X_0, \ldots , X_n]_d / (J_q)_d\rightarrow 
\mathbb{F}_q^N$ be the 
$\mathbb{F}_q$-linear transformation given by $\Theta( f + (J_q)_d ) = 
(f(\boldsymbol{\gamma}_1) 
\ldots, 
f(\boldsymbol{\gamma}_N) )$, where we write the points of 
$\mathbb{P}^n(\mathbb{F}_q)$ in the 
standard 
notation, i.e.\ the first nonzero entry from the left is equal to 1. The 
projective generalized Reed-Muller code of order $d$, denoted by $PGRM_q(n, 
d)$, is the 
image of $\Theta$.
\end{definition}
It is easy to check that if one chooses another representation for the 
projective points the code thus obtained is equivalent to the code defined 
above. It is also easy to prove that if $d \geq n(q - 1) + 1$ then $\Theta$ is 
an isomorphism, so the relevant range to investigate the parameters of $PGRM$ 
codes is $1 \leq d \leq n(q - 1)$.

Lachaud, in \cite{lachaud2} presents some bounds for $\delta_{PGRM_q(n, d)}$, 
the minimum distance for 
$PGRM_q(n, d)$, and determines the true value in a special case. Serre, in 1989 
(see \cite{serre}), determined the minimum distance of $PGRM_q(n, d)$ when $d < 
q$.  In 1990 Lachaud (see\cite{lachaud}) presents some properties that some 
higher weights of $PGRM_q(n, d)$ must have, when $d \leq q$ and $d \leq n$.

Let $g \in \mathbb{F}_q[X_1, \ldots, X_n]$ be a polynomial of degree $d - 1 \geq 1$ and let $\omega$ be the Hamming weight of $\Phi(g + I)$. Let 
$g^{(h)}$ be the 
homogenization of $g$ with respect to $X_0$, then the degree of $X_0 g^{(h)}$ 
is $d$ 
and the weight of $\Theta(X_0 g^{(h)} + (J_q)_d)$ is $\omega$. In particular
$\delta_{PGRM_q(n, d)} \leq \delta_{GRM_q(n, d - 1)}$. When $d = 1$ all the 
codewords of $PGRM_q(n, d)$ have the same number of zeros entries (hence the 
same weight), which is equal to the number of points of a hyperplane in 
$\mathbb{P}^n(\mathbb{F}_q)$, this also implies that for $d = 1$ there are no 
higher Hamming weights. In 1991 S{\o}rensen (see \cite{sorensen}) proved that 
$\delta_{PGRM_q(n, d)} = \delta_{GRM_q(n, d - 1)}$ holds for all $d$ in the 
relevant range. After this paper, similarly to what had happened with GRM 
codes, the subject lay dormant for almost two decades. Then, in 2007 Rodier and 
Sboui (see \cite{rod-sboui}), under the condition $d(d-1)/2 < q$  determined a 
Hamming weight of $PGRM_q(n, d)$, which is not the minimal and is only achieved 
by codewords whose zeros are hyperplane arrangements. In 2008 the same authors 
(see \cite{rod-sboui2}) proved that for $q/2 + 5/2 \leq d < q$ 
the third weight of PGRM is not only achieved by evaluating $\Theta$ in the 
classes of totally 
decomposable 
polynomials but can also be obtained in this case from classes of some 
polynomials 
having an irreducible quadric as a factor. Also in 2008, Rolland (see 
\cite{rolland2}) proved the 
equivalent of Delsarte, Goethals and Mac Williams theorem for PGRM codes, 
completely characterizing the codewords of $PGRM_q(n, d)$ which have minimal 
weights, and proving that they only arise as images by $\Theta$ of classes of 
totally 
decomposable polynomials, which in a sense may be thought of as the 
homogenization of the polynomials described by Delsarte et al. In 2009 
Sboui (\cite{sboui2}) determined the second and third weights of 
$PGRM_q(n, d)$ in the range $5 \leq d \leq q/3 + 2$. He proved that codewords 
which have these weights come only from evaluation of classes of totally 
decomposable polynomials and calculated the number of codewords having   
weights equal to  the minimal distance, or  the second weight, or the third weight.
In the already mentioned paper of 2014 (see \cite{ballet}), Ballet and Rolland 
we find another proof of Rolland's result on minimal weight codewords of PGRM. 
They also present lower and upper bounds for the second weight of $PGRM_q(n, 
d)$.

Putting together the reasoning presented in the beginning of the preceding 
paragraph and S{\o}rensen's result $\delta_{PGRM_q(n, d)} = \delta_{GRM_q(n, d 
- 1)}$, and writing $\delta^{(2)}_{PGRM_q(n, d)}$ for the second Hamming weight 
of $PGRM_q(n, d)$, we get $\delta^{(2)}_{PGRM_q(n, d)} \leq 
\delta^{(2)}_{GRM_q(n, d - 1)}$ 
for 
all $2 \leq d \leq n(q - 1) + 1$. In 2016 Carvalho and Neumann (see 
\cite{car-neu2016}) determined the 
second weight of $PGRM_2(n, d)$ for all $d$ in the relevant range, and in 2018 
(see \cite{car-neu2018}) they also determined the second weight of $PGRM_q(n, 
d)$, for $q \geq 3$ and 
almost all values of $d$. For some values of $d$, in both papers, it happened 
that $\delta^{(2)}_{PGRM_q(n, d)} < \delta^{(2)}_{GRM_q(n, d - 1)}$, and they 
proved 
that in all these cases the zeros of the codewords with weight 
$\delta^{(2)}_{PGRM_q(n, d)}$ are not hyperplane arrangements. 
They also observed that, writing $d - 1 = k (q - 1) + \ell$, with $0 \leq k 
\leq n - 1$ and $0 < \ell \leq  q - 1$, in the case where $q = 3$, $k > 0$  and 
$\ell = 1$ we have $\delta^{(2)}_{PGRM_q(n, d)} = \delta^{(2)}_{GRM_q(n, d - 
1)}$ and there are codewords of weight $\delta^{(2)}_{PGRM_q(n, d)}$ whose 
set of zeros are hyperplane arrangements and others which do not have this 
property. The tables below show the current results for 
$\delta^{(2)}_{PGRM_q(n, d)}$, where we write $d - 1 = k (q - 1) + \ell$ as 
above. The tables also present the values of 
$\delta^{(2)}_{GRM_q(n, d - 1)}$ so the reader can see the cases where one has 
 $\delta^{(2)}_{PGRM_q(n, d)} < \delta^{(2)}_{GRM_q(n, d - 1)}$.

\begin{table}[htbp]
	\centering
	\begin{tabular}{ccccc}\hline\hline
		$n$ & $k$   & $\ell$ 
		& $\delta^{(2)}_{GRM_2(n, d - 
		1)}$ & $\delta^{(2)}_{PGRM_2(n, d - 
		1)}$ 
		 \\
		\hline\hline
		$n \ge 3 $    & $k=0$ & $\ell =1$
		&	$2^n$	 & $3 \cdot 2^{n-2}$  \\
		$n\ge 4$ & $1 \le k < n-2$ & $\ell =1$
		&	$3\cdot 2^{n-k-2}$	 & $3 \cdot 2^{n-k-2}$\\
		$n\ge 2$ & $k=n-2$ & $\ell =1$
		&	$4$	 & $4$  \\
        $n\ge 2$ & $k=n-1$ & $\ell =1$
        		&	$2$	 & $2$  \\
		\hline\hline
	\end{tabular}
	\caption{Second (or next-to-minimal) weights for $GRM_q(n, d)$ and 
	$PGRM_q(n, d)$ when  
	$n\geq 2$ and $q=2$}
	\label{pesoq2}
\end{table}
\begin{table}[htbp]
	\centering
	\begin{tabular}{ccccc}\hline\hline
		$n$ & $k$   & $\ell$ 
		& $\delta^{(2)}_{GRM_q(n, d - 
				1)}$ & $\delta^{(2)}_{PGRM_q(n, d - 
				1)}$ \\
		\hline\hline
		$n=2$ & $k=0$ & $\ell =1$
		&	$3^2$	 & $3^2$   \\
		$n \ge 3$ & $k=0$ & $\ell =1$
		&	$3^n$	 & $8 \cdot 3^{n-2}$   \\
		$n \ge 3$ & $1 \leq k \leq n-2$ & $\ell =1$
		&	$8\cdot 3^{n-k-2}$	 & $8\cdot 3^{n-k-2}$ \\
		$n \ge 2$ &  $0 \leq k \leq n-2$ & $\ell =2$
		&	$4\cdot 3^{n-k-2}$	 & $4\cdot 3^{n-k-2}$ \\
		$n \ge 1$ &  $k=n-1$ & $\ell =1, 2$
		&	$4 - \ell$	 &  $4 - \ell$ \\
		\hline\hline
	\end{tabular}
	\caption{Second (or next-to-minimal) weights for $GRM_q(n, d)$ and 
	$PGRM_q(n, d)$ when  
	$n\geq 1$ and $q=3$}
	\label{pesoq3}
\end{table}
\begin{table}[htbp]
	\centering
	\begin{tabular}{ccccc}\hline\hline
		$n$ & $k$   & $\ell$ 
		& $\delta^{(2)}_{GRM_q(n, d - 
						1)}$ & $\delta^{(2)}_{PGRM_q(n, d - 
						1)}$ \\
		\hline\hline
		$n=2$ & $k=0$ & $\ell =1$
		&	$q^2$	 & $q^2$ \\
		$n\geq 3$ & $k < n-2$ & $\ell =1$
		&	$q^{n-k}$	 & $q^{n-k} - q^{n-k-2}$\\
		$n \geq 3$ & $k = n -2$ & $\ell =1$
		&	$q^2$	 & Unknown \\
		$n \geq 2$ & $k \le n -2$ & $1 < \ell \le \frac{q+1}{2}$
		&	$(q-1)(q-\ell 
		+ 1)q^{n-k-2}$	 &  $(q-1)(q-\ell + 1)q^{n-k-2}$ \\
		$n \geq 2$ & $k \le n -2$ & $\frac{q+1}{2} < \ell \leq q -1 $ 
		&	$(q-1)(q-\ell + 1)q^{n-k-2}$	 & Unknown  \\
		$n \ge 1$ & $k = n -1$ & $1 \leq \ell \leq q-1$
		&	$q-\ell + 1$	 &   $q-\ell + 1$ \\
		\hline\hline
	\end{tabular}
	\caption{Second (or next-to-minimal) weights for $GRM_q(n, d)$ and 
	$PGRM_q(n, d)$ when  
	$n\geq 1$ and $q \geq 4$}
	\label{pesonovo}
\end{table} 
\noindent
A generalization of PGRM codes was introduced in 2017 by  Carvaho, Neumann and 
L\'{o}pez (see 
\cite{car-neu-lop}), as the class of codes called ``projective nested 
cartesian codes''. They determined the dimension of these codes, bounds for the 
minimum distance and the exact value of this distance in some cases.

In the present paper we extend Delsarte, Goethals and Mac Williams theorem to 
the class of affine cartesian codes $\mathcal{C}_\mathcal{X}(d)$  defined 
above, in the case where the sets 
$K_1 \subset \cdots \subset K_n$ are subfields of $\mathbb{F}_q^n$. 
Our main results are Proposition \ref{minimalcodewords1} , Proposition 
\ref{pesominimo} and Theorem 
\ref{pesominimofinal} which show that, as in the GRM 
codes, the minimal weight codewords of $\mathcal{C}_\mathcal{X}(d)$ come from 
the evaluation  of 
$\Psi$ in classes $f + I$ of polynomials $f$ which, after a suitable action of 
an  automorphism group, may be written as the product of certain degree 
one polynomials. In the next section we introduce the concept of code as an 
$\mathbb{F}_q$-vector space of functions (following \cite{erickson} and 
\cite{leducq}) and define the relevant automorphism group for the main result.
We then study
the intersection of certain affine subspaces of 
$\K^n$ with $\X$ to find information on the structure of functions that 
have ``few'' points in the support (see Corollary \ref{fatores}). Then, in 
the beginning of 
Section 3, we use these results to determine the structure of the functions (or 
codewords) of minimal weight, for $d$ within a certain range -- in a sense, for 
the lower values of $d$ (see Proposition \ref{minimalcodewords1}). Finally, 
after exploring a little further the properties of the intersection 
of certain  hyperplanes with $\X$, we prove our main result (see Theorem 
\ref{pesominimofinal}) which generalizes the result by 
Delsarte,
Goethals and Mac Williams.

\section{Preliminary results}
Let $\mathcal{C}_\mathcal{X}(d)$ be the affine cartesian code as in Definition 
\ref{def1}. We assume from now on that $K_1, \ldots, K_n$ are fields and 
that $K_1 \subset K_2 \subset \cdots \subset K_n \subset \K$. Recall that 
$| K_i | = d_i$ for $i = 1,\ldots, n$, so $I_\X = (X_1^{d_1} - X_1, \ldots, 
X_n^{d_n} - X_n )$, and observe 
that, since $\Psi$ is an isomorphism, the code $\mathcal{C}_\mathcal{X}(d)$ is
isomorphic to the $\K$-vector space of the classes of polynomials in $\K[X_1,
\ldots, X_n]/I_{\mathcal{X}}$ of degree up to $d$ (together with the zero
class).
It is well known that, given a subset $Y \subset \K^n$, any function $f : Y
\rightarrow \K$ is given by a polynomial $P \in \K[X_1, \ldots, X_n]$ (again, 
this
is a consequence of the fact that given $\balpha \in \K^n$ there exists a 
polynomial
$P_{\balpha} \in \K[X_1, \ldots, X_n]$ such that $P_{\balpha}(\balpha) = 1$ and 
$P_{\balpha}(\boldsymbol{\beta}) = 0$
for any $\boldsymbol{\beta} \in \K^n \setminus \{ \balpha \}$). Denoting by 
$C_\mathcal{X}$ the
$\K$-algebra of functions defined on $\mathcal{X}$ we clearly have an 
isomorphism
$\Phi: \K[X_1, \ldots, X_n]/I_{\mathcal{X}} \rightarrow C_\mathcal{X}$
hence for each function $f \in  C_\mathcal{X}$ there exists a unique polynomial 
$P \in \K[X_1, \ldots, X_n]$ such that the degree of $P$ in the
variable $X_i$ is less than $d_i$ for all $i = 1, \ldots, n$, and $\Phi(P + 
I_{\mathcal{X}}) = f$.
%
%
%
%
%
\begin{definition} We say that 
 $P$ is {\em the reduced polynomial associated to  $f$} and we  define the {\em 
 degree of $f$} as being the degree of $P$.
\end{definition}

We denote by $C_\mathcal{X}(d)$ the $\K$-vector space formed by functions of 
degree up to $d$, together with the zero function. We saw above that 
$C_\mathcal{X}$ is isomorphic to 
$\K[X_1, \ldots, X_n]/I_{\mathcal{X}}$, and hence to $\K^m$, and 
clearly $C_\mathcal{X}(d) \subset C_\mathcal{X}$ is isomorphic to the code 
$\mathcal{C}_\mathcal{X}(d) \subset \K^m$, so from now on we also call 
$C_\mathcal{X}(d)$ the affine cartesian code of order $d$. To study the 
 codewords of minimum weight we define the support of a function $f \in 
C_\mathcal{X}$ as the 
set
$\{ \negrito{\alpha} \in \mathcal{X} \mid f(\negrito{\alpha}) \neq 0 \}$ and we write 
$|f|$ for its cardinality, which, in this approach, is the Hamming weight of $f$. Thus the minimum distance of $C_\mathcal{X}(d)$ is  
$\delta_\mathcal{X} (d) := 
\min \{ |f| \mid  f \in C_\mathcal{X} (d) \text{ and } f \neq 0 \}$.
We denote by  
\[
Z_\mathcal{X}(f) := \{ \balpha \in \mathcal{X} \mid f(\balpha)=0 \}
\]
the set of zeros of $f \in C_\mathcal{X}$, and given functions $g_1, \ldots, 
g_s$ defined on $\K^n$ we denote by  $Z(g_1,\ldots , g_s)$ be the set 
of common zeros, in $\K^n$, of these functions.

We write $\aff(n,\K)$ for the affine group of $\K^n$, i.e.\ the 
transformations 
of
$\K^n$ of the type  $\balpha \longmapsto A \balpha + \boldsymbol{\beta}$, where 
$A \in GL(n,\K)$
and $\boldsymbol{\beta} \in \K^n$. 
\begin{definitions}
The affine group associated to 
$\mathcal{X}$ is
\[
\aff(\mathcal{X}) =
\{ \varphi: \X \rightarrow \X \mid \varphi = \psi_{|_{\X}} \textrm{ with } 
 \psi \in \aff(n,\K) \text{ and }
\psi(\mathcal{X})=\mathcal{X} \}.
\]

We say that $f, \, g \in C_\mathcal{X}$ are $\mathcal{X}$-equivalent if there exists
$\varphi \in \aff(\mathcal{X})$
such that  $f = g \circ \varphi$.

An affine subspace $G \subset \K^n$ of dimension $r$ is said to be 
 $\mathcal{X}$-affine if there exists 
$\psi \in \aff(n,\K)$ and 
$1 \le i_1 < \cdots < i_r \le n$ such that  $\psi(\X) = \X$ and 
$\psi( \langle e_{i_1}, \ldots , e_{i_r}\rangle ) = G$, where we write  
$\{e_1, \ldots, e_n\}$ for the canonical basis of $\K^n$. We denote by 
$x_i$ the coordinate function $x_i (\sum_j a_j e_j) = a_i$ where $\sum_j a_j 
e_j 
\in \mathbb{F}_q^n$  (and by abuse of notation we also denote by $x_i$ its 
restriction to $\X$) for all $i =1, \ldots, n$.
Let $f \in C_\X$ be a  reduced polynomial of degree one,  if 
there exists 
$\varphi \in \aff(\mathcal{X})$ and $ i \in \{1, \ldots, n\}$  such that $x_i 
\circ \varphi = f$ on the points of $\X$ then we say that $f$ is $\X$-linear.
\end{definitions}

%
%
%

\noindent

\noindent
Let $\{i_1, \ldots, i_s\} \subset \{1, \ldots, n\}$ and $j \in \{1, \ldots, 
n\}$,  we define
$ \mathcal{X}_{i_1, \ldots , i_s} := K_{i_1} \times \cdots \times K_{i_s}$,  
and  \\
$\mathcal{X}_{\widehat{j}} := K_{1} \times \cdots \times K_{j-1}  \times 
K_{j+1} 
\times \cdots  \times K_n$.

\begin{definition} \label{f_alpha}
Let $j \in \{1, \ldots, n\}$, 
for every $\alpha \in K_j$ we have an evaluation homomorphism of $\K$-algebras 
given by
$$
\begin{array}{ccl}
C_\mathcal{X}        & \longrightarrow & C_{\mathcal{X}_{\widehat{j}}} \\
f                    & \longmapsto     &
f(x_1,\ldots, x_{j-1},\alpha,x_{j+1}, \ldots , , x_n) =: f_{\alpha}^{(j)}\, .
\end{array}
$$
\end{definition}
%

We now present two results which we will freely use in what follows. The first 
one states the value of the minimum distance of $C_{\mathcal{X}}(d)$.

\begin{theorem}\cite[Thm.\ 3.8]{lopez-villa} \label{min-dist}
The minimum distance $\delta_{\mathcal{X}}(d)$ of $C_{\mathcal{X}}(d)$ is 1,
if $d \geq \sum_{i = 1}^n (d_i - 1)$, and for $1 \leq d < \sum_{i = 1}^n (d_i -
1)$ we have 
	\[\delta_{\mathcal{X}}(d) = (d_{k + 1} - \ell ) \prod_{i = k + 2}^n d_i \]
	where $k$ and $\ell$ are uniquely defined by $d = \sum_{i = 1}^k (d_i - 1) +
\ell$ with $0 < \ell \leq d_{k + 1} - 1$ (if $k + 1 = n$ we understand that
$\prod_{i = k + 2}^n d_i = 1$, and if $d < d_1 - 1$ then we set $k = 0$ and
$\ell = d$).
\end{theorem}

The second one is a very useful numerical result, closely related to the above theorem (the link between these two results is explained in \cite{carvalho}). 

\begin{lemma}\cite[Lemma 2.1]{carvalho}\label{2.1}
	Let $ 0 < d_1 \leq \cdots \leq d_n$ and $1 \leq d \leq \sum_{i = 1}^n (d_i -
1)$ be integers. Let $m(a_1, \ldots , a_n) = \prod_{i = 1}^n (d_i - a_i)$, where
$0 \leq a_i < d_i$ is  an integer for all $i = 1,\ldots, n$. Then
	\[ 
	\min \{ m(a_1, \ldots, a_n) \, | \, a_1 + \cdots + a_n \leq d \} =  (d_{k + 1}
- \ell) \prod_{i = k + 2}^n d_i
	\]
	where $k$ and $\ell$ are uniquely defined by $d = \sum_{i =1}^k (d_i -1) + \ell
$,  with $0 < \ell \leq d_{k + 1} - 1$ (if $s < d_1 - 1$ then take $k = 0$ and
	$\ell = d$, if $k + 1 = n$ then we understand that $\prod_{i = k + 2}^n d_i =
1$). 
\end{lemma}

From Theorem \ref{min-dist} we get that the relevant range for $d$ is $1 \leq d 
< \sum_{i = 1}^n (d_i - 1)$ (the case $d = 0$ is trivial and if $d \geq \sum_{i 
= 1}^n (d_i -
1)$ we have $C_\mathcal{X}(d) \cong \K^m$). In what follows we will always 
assume that $1 \leq d < \sum_{i = 1}^n (d_i - 1)$ and will also freely use 
the decomposition $d = \sum_{i =1}^k (d_i -1) + \ell
$,  with $0 < \ell \leq d_{k + 1} - 1$ (and $0 \leq k < n$). In many places we 
consider a nonzero function $g$ 
defined in $\mathcal{X}_{i_1, \ldots , i_s} \subset \K^{s}$
which belongs to $C_{\mathcal{X}_{i_1, \ldots , i_s}}(d)$, 
and we want to 
estimate $|g|$. 
Applying Theorem \ref{min-dist} 
we 
get that $|g| \geq 1$ if $d \geq \sum_{t = 1}^{s} (d_{i_t} - 1)$
while 
if   
$d < \sum_{t = 1}^{s} (d_{i_t} - 1)$ then $|g| \geq 
\delta_{\mathcal{X}_{i_1, \ldots , i_s}}(d)$, and we find 
$\delta_{\mathcal{X}_{i_1, \ldots , i_s}}(d)$ by a proper application of 
the formula in Theorem \ref{min-dist}.
 Since 
$\delta_{\mathcal{X}_{i_1, \ldots , i_s}}(d) = 1$ in the case 
where $d \geq 
\sum_{t = 1}^{s} (d_{i_t} - 1)$, we can always write $| g 
| \geq 
\delta_{\mathcal{X}_{i_1, \ldots , i_s}}(d)$.

%
%

The following result shows that functions which are related by an 
affine transformation have the same degree. 

\begin{lemma}\label{grau}
Let $\varphi \in \aff(\mathcal{X})$ and $f \in C_\mathcal{X}$ with $f \neq 0$, 
then
 $\deg f = \deg ( f \circ \varphi)$.
\end{lemma}
\begin{proof} Since $\varphi \in \aff(\mathcal{X})$ we have that 
$\varphi(\boldsymbol{\alpha}) = A \boldsymbol{\alpha} + \boldsymbol{\beta}$
where $A \in GL(n,\K)$
and $\boldsymbol{\beta} \in \K^n$. Let $P \in \K[X]$ be the reduced polynomial 
associated to $f$, and let's endow $\K[X]$ with a degree-lexicographic order. 
Then the reduced polynomial associated to $f \circ \varphi$ is the remainder, 
say $Q$, in the division of $P(A \boldsymbol{X} + \boldsymbol{\beta})$ by $\{ 
X_1^{d_1} - X_1, \ldots, X_n^{d_n} - X_n \}$, where $\boldsymbol{X}$ 
is a 
column vector with entries equal to $X_1, \ldots, X_n$. Thus $\deg Q \leq \deg 
P(A \boldsymbol{X} + \boldsymbol{\beta}) \leq \deg P$, so that 
$\deg (f \circ \varphi) \leq \deg f$. Applying the argument 
to 
$\varphi^{-1}$ we conclude that $\deg (f \circ \varphi) = \deg 
f$.
\end{proof}

The next result, although simple, is the basis for many 
important results that follow.

\begin{lemma}\label{divide}
Let $f,h \in C_\mathcal{X}$ be nonzero functions.
There exists a function  $g \in C_\mathcal{X}$ such that $f = gh$ if and only 
if 
$Z_\mathcal{X}(h) \subset Z_\mathcal{X}(f)$, i.e.\ 
$h$ is  a factor of $f$ if and only of $f$ vanishes in $Z_\mathcal{X}(h)$.
Moreover, if $h$ is $\mathcal{X}$-linear then  $\deg g = \deg f - 1$.
\end{lemma}
\begin{proof}
If $f=gh$  and $h(\negrito{\alpha})=0$ 
then $f(\negrito{\alpha})=0$, for all   $\balpha \in \mathcal{X}$.
Assume now that 
$Z_\mathcal{X}(h) \subset Z_\mathcal{X}(f)$, and let $g: \X \rightarrow 
\mathbb{F}_q$ be defined by $g(\balpha) = 0$ if $\balpha \in Z_\mathcal{X}(h)$, 
and $g(\balpha) = f(\balpha)/h(\balpha)$ if $\balpha \in \X \setminus 
Z_\mathcal{X}(h)$, then clearly $f = g h$ as functions of $C_\mathcal{X}$.

Let's assume now that $h \mid f$ and that $h$ is 
$\mathcal{X}$-linear, so that  
$h \circ \varphi = x_i$  for some $i \in \{ 1, \ldots , n\}$ and 
$\varphi \in \aff(\mathcal{X})$.  Then $f \circ \varphi = (g \circ \varphi)
(h \circ \varphi)$ and since from Lemma \ref{grau} $\deg f = \deg (f \circ 
\varphi)$ we may simply assume that $h = x_i$.
Let $P$ be the reduced  polynomial associated to  $f$ and write  $P = X_i \cdot 
Q + R$, where $Q, R \in \K[X_1, \ldots, X_n]$ 
and $X_i$ does not appear in any monomial of $R$. Observe 
that for any $j \in \{1,
\ldots, n\}$, the degree of $X_j$ in any monomial of $Q$ is 
at most $d_j
- 1$. Let $g$ and $t$ be the functions associated to $Q$ and $R$, respectively,
so $f = x_i g + t$.  We must have $t = 0$, otherwise $t(\balpha) \neq 0$ for
some $\balpha = (\alpha_1, \ldots, \alpha_n) \in \X$, hence taking
$\tilde{\balpha} = (\tilde{\alpha}_1, \ldots, \tilde{\alpha}_n)$, with
$\tilde{\alpha}_j = \alpha_j$ for $j \in \{1, \ldots, n\} \setminus \{ i \}$ and
$\tilde{\alpha}_i = 0$ we get $x_i(\tilde{\balpha}) = 0$ hence
$f(\tilde{\balpha}) = 0 $ but $t(\tilde{\balpha}) \neq 0$, a contradiction.
Since $R$ is the reduced polynomial associated to $t$ we get $R = 0$, and since 
$Q$ is the reduced polynomial of $g$ we get $\deg g = \deg Q = \deg f - 1$.
\end{proof}

\begin{lemma}\label{divide2}
Let $h$ be a nonzero function in $C_\mathcal{X}(d)$
such that for some $i \in \{1, \ldots, n\}$ and some $\varphi \in 
\aff(\mathcal{X})$ we have $h = x_i \circ \varphi$. Then, for $\alpha \in \K$, 
we get that 
$h - \alpha$ is $\mathcal{X}$-linear if and only if  $\alpha \in K_i$.
Moreover, let $f \in C_\mathcal{X}(d)$, $f \neq 0$ and let $\alpha_1, \ldots,
\alpha_s$ be distinct elements of $K_i$ such that $Z_\mathcal{X}(h - \alpha_j)
\subset Z_\mathcal{X}(f)$ for all $j = 1, \ldots, s$, then there exists 
$g \in C_\mathcal{X}(d - s)$ such that 
$f = g \cdot \displaystyle \prod_{j=1}^{s} (h - \alpha_j)$.
\end{lemma}
%
%
\begin{proof} 
Assume that $\alpha \in K_i$ and consider the affine transformation 
$\tilde{\varphi}: \K^n \rightarrow \K^n$ given by $\tilde{\varphi}(\balpha) = 
\varphi(\balpha) - \alpha e_i$ for all $\balpha \in \K^n$, then one can easily 
check that $\tilde{\varphi} \in \aff(\X)$ and $x_i \circ \tilde{\varphi} = h - 
\alpha$. On the other hand, suppose that $h - \alpha$ is $\X$-linear, then $h - 
\alpha$ must vanish on some point of $\X$. From $h = x_i \circ \varphi$ we get 
that $h(\X) \subset K_i$ so we must have $\alpha \in K_i$.

Since $h - \alpha_1$ is 
$\X$-linear 
and $Z_\mathcal{X}(h - \alpha_1) \subset Z_\mathcal{X}(f)$ then from Lemma 
\ref{divide} we get that $f =  g_1 (h - \alpha_1)$ with $g_1 
\in 
C_\mathcal{X}(d-1)$. If $s = 1$ we're done, if $s \geq 2$ then 
 from $Z_\mathcal{X}(h - \alpha_2) \subset 
Z_\mathcal{X}(f)$ and the fact that $Z_\mathcal{X}(h - \alpha_1) \cap 
Z_\mathcal{X}(h - \alpha_2) = \emptyset$ we get that $Z_\mathcal{X}(h -  
\alpha_2) \subset Z_\mathcal{X}(g_1)$. From the hypothesis 
and Lemma 
\ref{divide} we get that $g_1 = g_2 (h - \alpha_2)$ with $g_2 
\in 
C_\mathcal{X}(d-2)$, this proves the statement in the case where $s = 2$ and  
if $s 
> 2$ the assertion is proved after a finite number of similar 
steps.
\end{proof}

If $G$ is $\mathcal{X}$-affine and there exists
$\psi \in \aff(n,\K)$ and  
$1 \le i_1 < \cdots < i_r \le n$ such that  $\psi(\X) = \X$ and 
$\psi( \langle e_{i_1}, \ldots , e_{i_r}\rangle ) = G$
then 
$\X_G := \mathcal{X}_{i_1, \ldots , i_r}$. 
The following results states an important property of the 
support of functions.

\begin{lemma} \label{suporte}
Let $f \in C_{\mathcal{X}}(d)$ be a nonzero function and let $S$ be its 
support. 
Then for every $\X$-affine subspace $G \subset \K^n$
of  dimension $r$, with $r \in \{1, \ldots, n-1\}$,  either $S \cap G = 
\emptyset$ or
$|S \cap G| \ge \delta_{\mathcal{X}_G}(d)$.
\end{lemma}
%
%
\begin{proof}
Since $G$ is an $\X$-affine subspace of dimension $r$ there exists 
an affine transformation $\psi: \K \rightarrow \K$
such that  $\psi(\mathcal{X})=\mathcal{X}$ and $G = \psi(V)$ where 
$V = \langle e_{i_1}, \ldots , e_{i_r} \rangle$. Observe that $\psi$ establishes
a bijection between the points of $V \cap \psi^{-1}(S)$ and $G \cap S$, we also
have that $\psi^{-1}(S)$ is the support of the function $f \circ \psi_{|_{\X}}$
which belongs to $C_{\mathcal{X}}(d)$ because $\deg f = \deg (f \circ
\psi_{|_{\X}})$. This shows that, for simplicity, we may assume that 
$G = 
\langle e_{i_1}, \ldots , e_{i_r} \rangle$. Suppose that $S \cap G \neq
\emptyset$ and let $P$ be the reduced polynomial associated to $f$, then 
  $f$
induces a nonzero function $\tilde{f}$ defined over $\X_G = 
\X_{i_1, \ldots , i_r} \subset \K^r$ 
whose reduced polynomial is $\tilde{P}(X_{i_1}, \ldots, 
X_{i_s})$ obtained from
$P$ by making $X_i = 0$ for all $i \in \{1, \ldots, n\} 
\setminus \{i_1, \ldots , i_s\}$. Clearly $\deg \tilde{f} 
\leq d$ so that $\tilde{f} \in C_{\X_G}(d)$,
also $|S \cap G| = | \tilde{f}|$ and as a consequence of Theorem \ref{min-dist}
we get $| \tilde{f}| \geq \delta_{\mathcal{X}_G}(d)$. 
\end{proof}

Observe, in the next result, that if $S$ is the support of a 
function then, from 
the above result, it already has property (2).

%
%
\begin{proposition} \label{desigualdadeS}
	Let $1 \leq d < \sum_{i = 1}^n (d_i - 1)$ and write 
$d = \sum_{i = 1}^k (d_i - 1) +
\ell$ as in Theorem \ref{min-dist}.	
	Let $S \subset \mathcal{X}$ be a nonempty set and
assume that $S$ has the following properties:
	\begin{enumerate}
		\item \label{propS1} $|S| < 
		\left(
		1 + \dfrac{1}{d_{k+1}}
		\right)
		\delta_\mathcal{X}(d) =
		\left(
		1 + \dfrac{1}{d_{k+1}}
		\right)
		(d_{k+1} - \ell) d_{k+2} \cdots d_n 
		$.
		\item \label{propS2} For every $\X$-affine subspace $G \subset \K^n$
of  dimension $r$, with $r \in \{0, \ldots, 
n-1\}$,  
either $S \cap G = 
\emptyset$ or 
		$|S \cap G| \ge 
		\delta_{\mathcal{X}_G}(d)$.
	\end{enumerate}
	Then there  exists an affine  subspace $H \subset \K^n$, of dimension 
$n-1$ and a transformation  $\psi \in \aff(n,\K)$  such that $\psi(\X) = \X$, 
$\psi(V_{k+1}) = H$ where $V_{k+1}$ is the $\K$-vector space generated by 
$\{e_1, \ldots, e_n\} \setminus \{e_{k + 1} \}$ (so, in particular, $H$ is 
$\X$-affine) and $S \cap H = \emptyset$.
%
\end{proposition}
%
%

%
%
\begin{proof}
We proceed by induction on $n$. When $n = 1$ we have $k = 0$, and from the hypothesis we get that  	
 $|S| < (1 + \dfrac{1}{d_1})(d_{1} - \ell) \le d_1 - \dfrac{1}{d_1}$, hence 
 $|S| \leq d_1 - 1$ and $S \subsetneqq K_1 \subset \K$. A $0$-dimensional 
 $\X$-affine subspace is just an element of $K_1$, so it is enough to take $H$ 
 as a point of $K_1 \setminus S$. 

Assume now that the statement is true for all $n < N$, and let $S  \subset \X 
\subset \K^N$ as in the hypothesis. For $\alpha \in K_{k+1}$ let
$$
G_\alpha = \alpha e_{k+1} + V_{k+1} =
\left\{ 
\negrito{\beta} \in \K^N \mid \negrito{\beta} = (\beta_1, 
\ldots, \beta_N) 
\textrm{ and }   \beta_{k+1} = \alpha 
\right\} 
$$

If for some $\alpha \in K_{k+1}$ we get $S \cap G_\alpha = \emptyset$ then we're done,  so assume from now on that $S \cap G_\alpha \neq \emptyset$ for all  
$\alpha \in K_{k+1}$. If $k=N-1$ we have $\delta_\mathcal{X}(d) = d_N - \ell$
and 
\[ 
	d_N \le \sum_{\alpha \in K_N}  |S \cap G_\alpha| = 
	|S| < 
	\left(
	1 + \dfrac{1}{d_N}
	\right) (d_N - \ell)
	\le 
	\left(
	1 + \dfrac{1}{d_N} 
	\right) (d_N - 1)
	= d_N - \dfrac{1}{d_N}\, ,
\]
a contradiction which settles this case.
%
%
%
Now we consider the case where $k \leq N - 2$. 
Since $G_\alpha$ is $\X$-affine we have
$|S \cap G_\alpha| \ge \delta_{\mathcal{X}_{\widehat{k+1}}} (d)
	= (d_{k+2} - \ell)d_{k+3} \cdots d_N$ for every $\alpha \in K_{k + 1}$.
Thus 
$d_{k+1} \delta_{\mathcal{X}_{\widehat{k+1}}} (d)
	\le |S| <
	\left(
	1 + \dfrac{1}{d_{k+1}}
	\right)
	\delta_\mathcal{X}(d)
$
and from the formulas for $\delta_{\mathcal{X}_{\widehat{k+1}}} (d)$ and 
$\delta_\mathcal{X}(d)$ we get $
	d_{k+1} (d_{k+2} - \ell) < 
	\left(
	1 + \dfrac{1}{d_{k+1}}
	\right) (d_{k+1} - \ell) d_{k+2}$.
Hence 
$$
	1 - \dfrac{\ell}{d_{k+2}} < 
	1 - \dfrac{\ell}{d_{k+1}^2} - \dfrac{\ell - 1}{d_{k+1}} \le 
	1 - \dfrac{\ell}{d_{k+1}^2}
$$
so that $d_{k+2} < d_{k+1}^2$. Assume that 
$K_{k+1} \subsetneq K_{k+2}$,  since this is a field extension 
we must have $d_{k+1}^2 \le d_{k+2}$, a contradiction which settles the case 
$k \leq N - 2$ and $d_{k + 1}< d_{k + 2}$. 	

The last case is when $k \leq N - 2$ and $d_{k + 1} = d_{k + 2}$, and now we will apply the induction hypothesis. To do that, for $\alpha \in K_{k+1}$, we 
consider the bijection $\xi_\alpha : G_\alpha \rightarrow \K^{N - 1}$ which 
acts on an $N$-tuple $\balpha \in G_\alpha$ by deleting the $(k+1)$-th entry 
(which is equal to $\alpha$). 
Observe that $\xi_\alpha$ establishes a bijection between affine subspaces of 
$\K^N$ contained in $G_\alpha$ and affine subspaces of $\K^{N - 1}$. 
Clearly $\mathcal{X}_{\widehat{k+1}} \subset \K^{N - 1}$ and we want to show 
that 
$\xi_\alpha(S \cap G_\alpha)$ has  property (2) of the statement (with 
$\X_{\widehat{k+1}}$ in place of $\X$). For this, let $L \subset \K^{N - 1}$ 
be  an 
$r$-dimensional 
$\X_{\widehat{k+1}}$-affine subspace. Then for some $\tilde{\psi} \in \aff(N - 
1, \K)$, 
given by $\tilde{\alpha} \mapsto \tilde{A}  \tilde{\alpha}+ \boldsymbol{\tilde{\beta}}$, with 
$\tilde{A} \in GL(N-1,\K)$
and $\boldsymbol{\tilde{\beta}} \in \K^{N-1}$, we have 
$\tilde{\psi}(\X_{\widehat{k+1}}) = \X_{\widehat{k+1}}$ and $\tilde{\psi}(L) = 
\langle \tilde{e}_{i_1}, \ldots, \tilde{e}_{i_r} \rangle$, where 
$\{\tilde{e}_{1}, \ldots, \tilde{e}_{N-1}\}$ is the canonical basis for 
$\K^{N-1}$.
%
%
We claim that $\xi_\alpha^{-1}(L)$ is an $\X$-affine subspace contained in 
$G_\alpha$ and to see that let $A$ be the matrix obtained from $\tilde{A}$ by 
adding an $N \times 1$ column of zeros as the $(k+1)$-th column, an $1 \times 
N$ line of zeros as the $(k+1)$-th line and changing the 0 at position $(k+1, 
k+1)$ to 1. Let $\boldsymbol{\beta}$ be the $N\times 1$ vector obtained from 
$\boldsymbol{\tilde{\beta}}$ by adding the entry $-\alpha$ at position $k+1$. 
Then, defining $\psi: \K^N \rightarrow \K^N$ by $\boldsymbol{\alpha} \mapsto A 
\boldsymbol{\alpha}  + 
\boldsymbol{\beta}$ we get that $\psi \in \aff(N, \K)$, and it is easy to check 
that $\psi(\X) = \X$ and that $\psi(\xi_\alpha^{-1}(L)) = \langle e_{j_1}, 
\ldots, e_{j_r} \rangle$, with $\{j_1, \ldots, j_r\} \subset \{1, \ldots, n\} 
\setminus \{k + 1\}$, $j_s = i_s$ whenever $i_s < k + 1$, and $j_s = i_s + 1$ 
whenever $i_s \geq k + 1$, for all $s = 1, \ldots, r$, so that $\{d_{j_1}, 
\ldots, d_{j_r}\} = \{d_{i_1}, \ldots, d_{i_r}\}$. To show that 
$\xi_\alpha(S \cap G_\alpha)$ has  property (2) of the statement, with 
$\X_{\widehat{k+1}}$ in place of $\X$, we observe that 
\[ 
|(\xi_\alpha(S \cap G_\alpha) \cap L | = |(S \cap G_\alpha) \cap 
\xi_\alpha^{-1}(L) | = | S \cap \xi_\alpha^{-1}(L)| \geq 
\delta_{\mathcal{X}_{j_1, \ldots, j_r}}(d) = 
\delta_{(\mathcal{X}_{\widehat{k+1}})_{i_1, \ldots, i_r}}(d).
\]
Now we prove that there exists $\alpha \in K_{k+1}$ such that $\xi_\alpha(S 
\cap G_\alpha)$ also has property (1), with $\X_{\widehat{k+1}}$ in place of 
$\X$. Indeed, if for all $\alpha \in K_{k+1}$ we have 	
	$$
	|\xi_\alpha(S \cap G_\alpha)| \ge \left(
			1 + \dfrac{1}{d_{k+2}} 
			\right)\delta_{\mathcal{X}_{\widehat{k+1}}}(d) = 
	\left(
		1 + \dfrac{1}{d_{k+2}}
		\right) (d_{k+2} - \ell) d_{k+3} \cdots d_N 	
	$$
then from $|\xi_\alpha(S \cap G_\alpha)| = |S \cap G_\alpha|$ we get 
$	|S| \ge  d_{k+1}	\left(
		1 + \dfrac{1}{d_{k+2}}
		\right)  (d_{k+2} - \ell) d_{k+3} \cdots d_N 
 =
	\left(
	1 + \dfrac{1}{d_{k+1}}
	\right)
	\delta_\mathcal{X}(d)\, 
	$ 	
(because $d_{k+1}=d_{k+2}$)	which contradicts property (1).
Thus, for some $\alpha \in K_{k+1}$ we get that $\xi_\alpha(S \cap G_\alpha) 
\subset \X_{\widehat{k+1}} \subset \K^{N - 1}$ satisfies properties (1) and 
(2), and from 
the induction hypothesis  there exists an $\X_{\widehat{k+1}}$-affine subspace 
$L \subset 
\K^{N-1}$ of dimension $N - 2$ and $\tilde{\psi} \in \aff(N - 1, \K)$ such 
that  
$\tilde{\psi}(\X_{\widehat{k+1}}) = \X_{\widehat{k+1}}$, 
$\psi(L)$ is the subspace generated by $\{\tilde{e}_1, \ldots,  
\tilde{e}_{N-1}\} \setminus \{\tilde{e}_{k + 1}\}$ and $\xi_\alpha(S \cap 
G_\alpha) \cap L = \emptyset$. From what we did above we get that 
$\xi_\alpha^{-1}(L)$ is an $(N - 2)$-dimensional $\X$-affine subspace of $\K^N$ 
and there exists $\psi \in \aff(N, \K)$ such that  
$\psi(\X) = \X$, $\psi(\xi_\alpha^{-1}(L))$ is the subspace generated by 
$\{{e}_1, \ldots,  {e}_{N}\} \setminus \{e_{k + 1}, e_{k + 2}\}$,  
 and $(S \cap G_\alpha) \cap \xi_\alpha^{-1}(L) = 
S \cap \xi_\alpha^{-1}(L) = \emptyset$. 
Thus $\psi(\xi_\alpha^{-1}(L))$ is the subvector space defined by $X_{k+1} = 0$ and $X_{k+2} = 0$, and
let $G_{(\gamma_1, \gamma_2)}$ be the hyperplane defined by the equation $\gamma_1 X_{k+1} + \gamma_2 X_{k + 2} = 0$, where $(\gamma_1 : \gamma_2) \in \mathbb{P}^1(K_{k + 1})$, observe that $G_{(\gamma_1, \gamma_2)} \cap G_{(\gamma'_1, \gamma'_2)} = \psi(\xi_\alpha^{-1}(L))$ whenever  $(\gamma_1 : \gamma_2) \neq (\gamma'_1 : \gamma'_2)$. One may easily check that for every $(\gamma_1 : \gamma_2) \in \mathbb{P}^1(K_{k + 1})$ there exists a linear transformation that takes $G_{(\gamma_1, \gamma_2)}$ onto the subspace defined by $X_{k+1} = 0$, so that $H_{(\gamma_1, \gamma_2)} := \psi^{-1}(G_{(\gamma_1, \gamma_2)})$ is an $\X$-affine subspace of dimension $N - 1$. We claim that for some $(\gamma_1 : \gamma_2) \in \mathbb{P}^1(K_{k + 1})$ we must have $S \cap H_{(\gamma_1, \gamma_2)} = \emptyset$. Indeed, if this is not true, then, since 
$H_{(\gamma_1, \gamma_2)} \cap H_{(\gamma'_1, \gamma'_2)} = \xi_\alpha^{-1}(L)$
(for any distinct pair $(\gamma_1 : \gamma_2), (\gamma'_1 : \gamma'_2),  \in \mathbb{P}^1(K_{k + 1})$) and $S \cap \xi_\alpha^{-1}(L) = \emptyset$ we get
	$$
	|S| \geq \sum_{(\gamma_1 : \gamma_2) \in \mathbb{P}^1({K_{k+1}})}  |S \cap 
	H_{(\gamma_1, \gamma_2)} | \geq (d_{k+1} + 1) 
	\delta_{\X_{\widehat{k+1}}}(d) = (d_{k+1} + 1) (d_{k+2} - \ell) d_{k+3} 
	\cdots d_N,
	$$
a contradiction with property (1) which, using $d_{k+1} = d_{k + 2}$, states that 
$$
	|S| < \left(
			1 + \dfrac{1}{d_{k+1}}
			\right)
			(d_{k+1} - \ell) d_{k+2} \cdots d_n = (d_{k+1} + 1) (d_{k+2} - \ell) d_{k+3} \cdots d_N,
$$
\end{proof}

The next result combines previous results and 
gives a first 
step in the direction of the main result.

%
%
%
%
\begin{corollary} \label{fatores}
Let $f$ be a nonzero function in $C_{\mathcal{X}}(d)$ such that 
$|f| < \left( 1 + \dfrac{1}{d_{k+1}} \right)
\delta_{\mathcal{X}}(d)$, then $f$ is a multiple of a 
function $h$ of degree 1 which is $\X$-equivalent to $x_{k+1}$.
\end{corollary}
%
%
\begin{proof}
Let $S$ be the support of $f$, from the hypothesis we have that $S$ has property (1) in the statement of Proposition \ref{desigualdadeS} and from Lemma 
\ref{suporte} we get that $S$ also has property (2). Thus, there exists 
an affine  subspace $H \subset \K^n$, of dimension 
$n-1$ and a transformation  $\psi \in \aff(n,\K)$  such that $\psi(\X) = \X$,
$\psi(V_{k+1}) = H$ with  
$V_{k+1} = \left\{ 
	 \negrito{\alpha} \in \K^n \mid \alpha_{k+1} = 0 
	\right\}$  and
$S \cap H = \emptyset$.
Hence $\psi^{-1}(S) \cap V_{k+1} = \emptyset$, and noting that 
$\psi^{-1}(S)$ is the support of 
the function $f \circ \psi_{|_{\X}} \in C_\X(d)$
we get that
$Z_\mathcal{X} (x_{k+1}) \subset Z_\mathcal{X} (f \circ \psi_{|_{\X}})$.
From Lemma \ref{divide} there exists
$g \in C_\X(d - 1)$ such that 
$f \circ \psi_{|_{\X}} = g \, x_{k+1}$, hence 
$f  = (g \circ \psi^{-1}_{|_{\X}}) \cdot (x_{k+1} \circ \psi^{-1}_{|_{\X}})$
and we can take $h = x_{k+1} \circ \psi^{-1}_{|_{\X}}$.
\end{proof}
%
%



Recall that we write $d = \sum_{i =1}^k (d_i -1) + \ell
$,  with $0 < \ell \leq d_{k + 1} - 1$ (and $0 \leq k < n$).

%
%
\begin{lemma}\label{desig_dist}  
	Let $f$ be a nonzero function in $C_\mathcal{X}(d)$,  and let $h \in C_{\X}(d)$ 
	be such that $h = x_j \circ \varphi$, where  $j \in \{1, \ldots, n\}$ and 
	$\varphi \in \aff(\X)$.  If  $m$ is the number of $\alpha \in K_j$ such that 
	$Z_\mathcal{X}(h - \alpha) \subset Z_\mathcal{X}(f)$ then $m \leq d$ and
	$|f| \ge (d_j - m ) \delta_{\mathcal{X}_{\widehat{j}}} (d - m)$.
\end{lemma}
%
%
\begin{proof}
	Let $\tilde{f}= f \circ \varphi^{-1}$, then $\tilde{f} \in C_\X(d)$, $f = 
	\tilde{f} \circ \varphi$ and 
	$\varphi$ 
	establishes a bijection between the sets  $Z_\mathcal{X}(h- \alpha)$ and 
	$Z_\mathcal{X}(x_j - \alpha)$ for all $\alpha \in K_j$, moreover we get that  
	$Z_\mathcal{X}(h- \alpha) \subset Z_\mathcal{X}(f)$ if and only if 
	$Z_\mathcal{X}(x_j- \alpha) \subset Z_\mathcal{X}(\tilde{f})$. This shows that, 
	in the statement, we can take $\varphi$ to be the identity transformation, 
	without loss of generality. Let $\alpha_1, \ldots, \alpha_m$ be the set of 
	elements $\alpha \in K_j$ such that $Z_\mathcal{X}(x_j- \alpha) \subset 
	Z_\mathcal{X}(f)$, from 
	Lemma \ref{divide2} we get that $f = g \cdot \displaystyle \prod_{i=1}^{m} (x_j 
	- \alpha_i)$, with $g \in  C_\mathcal{X}(d-m)$, and in particular $m \leq d$. 
	Observe that for all $\alpha \in K_j \backslash 
	\{\alpha_1 , \ldots , \alpha_m \}$
	we get $g_\alpha^{(j)} \neq 0$, so that
	$$
	|f| = \sum_{\alpha \neq \alpha_i} |g_\alpha^{(j)}|
	\ge (d_j - m ) \delta_{\mathcal{X}_{\widehat{j}}} (d - m) \, .
	$$ 
\end{proof}
%
%

For our purposes it is important to know when a  function 
$f \in C_\mathcal{X}(d)$ has minimal weight, i.e. when $| f | =   
\delta_{\mathcal{X}} (d)$. Taking into account the previous result, and using 
its notation, we 
investigate when $(d_j - m ) \delta_{\mathcal{X}_{\widehat{j}}} (d - m) \geq 
\delta_{\mathcal{X}} (d)$ holds, and under which conditions equality holds.

\begin{lemma}\label{desigparcial}
Let $1 \leq j \leq k+1$. If $d_j > d_{k+1} -\ell$,
for $0 < m < \ell+ (d_j - d_{k+1})$ we have
$$
(d_{j} - m ) \delta_{\mathcal{X}_{\widehat{j}}}
(d - m) > \delta_{\mathcal{X}} (d).
$$
\end{lemma}
\begin{proof}
Observe that we may write
$$
d - m = \sum_{i=1,i \neq j}^{k+1} (d_i - 1) + \ell - m + (d_j - d_{k+1})\, ,
$$
and note that 
$\ell - m + d_j - d_{k+1} \leq \ell - m < \ell < d_{k + 1} \leq d_{k + 2}$
so that  $\delta_{\mathcal{X}_{\widehat{j}}} (d - m) = 
(d_{k+2} - (\ell - m + d_j - d_{k+1})) 
\displaystyle
\prod_{i=k+3}^{n} d_i$.
From $\delta_{\mathcal{X}}(d) = (d_{k+1} - \ell) 
\prod_{i=k+2}^{n} d_i$ and 
$$
(d_j - m)	(d_{k+2} - (\ell - m + d_j - d_{k+1})) - (d_{k+1} - \ell)d_{k+2}  =
(\ell - m + d_j - d_{k+1}) 
(d_{k + 2} - d_j + m) > 0
$$
we get 
$$
(d_{j} - m ) \delta_{\mathcal{X}_{\widehat{j}}} (d - m) >
\delta_{\mathcal{X}} (d).
$$
\end{proof}

\begin{lemma}\label{numericoj} 
Let $1 \leq j \leq k$.
For $0 < m < d_{j}$ we have
$(d_{j} - m ) \delta_{\mathcal{X}_{\widehat{j}}}
(d - m)  \geq \delta_{\mathcal{X}} (d)$, with equality if and only if 
$m=d_j - 1$ or 
both $d_j > d_{k+1} - \ell$ and $m=\ell + d_j  - d_{k+1}$.
\end{lemma}
\begin{proof} By Lemma \ref{desigparcial}, we may consider
$\max\{1, \ell+ (d_j - d_{k+1})\} \leq m \leq d_j - 1$. 
In this case we write
$$
d - m = \sum_{i=1,i\neq j}^{k} (d_i - 1) + \ell + (d_j - 1 - m),
$$
and we observe that  $0 < \ell + d_j - 1 - m \leq d_{k+1} - 1$,
so that $\delta_{\mathcal{X}_{\widehat{j}}} (d - m)  =
(d_{k+1} - (\ell + d_j - 1 - m)) 
\prod_{i=k+2}^{n} d_i$. From 
$$
(d_j - m) (d_{k+1} - (\ell + d_j - 1 - m))
 - (d_{k+1} - \ell)  =
(m - (\ell + d_j  - d_{k+1}) ) 
(d_j -1 - m) \geq 0
$$
we get 
$$
(d_{j} - m ) \delta_{\mathcal{X}_{\widehat{j}}} (d - m) \geq
\delta_{\mathcal{X}} (d)\, ,
$$
with equality if and only if
$m=d_j - 1$ or 
both $\ell + d_j - d_{k+1} > 0$ and  $m=\ell + d_j  - d_{k+1}$.
\end{proof}

\begin{lemma}\label{numericok} 
For $0 < m < d_{k+1}$ we have
$(d_{k+1} - m ) \delta_{\mathcal{X}_{\widehat{k+1}}}
(d - m)  \geq \delta_{\mathcal{X}} (d)$, with equality if and only if 
$m=\ell$ or both $m=d_{k+1}-1$ and $d_k \geq d_{k + 1} - \ell$.
\end{lemma}
\begin{proof}
By Lemma \ref{desigparcial}, we may consider
$\ell \leq m \leq d_{k+1} - 1$. 
In this case 
we write 
$$
d - m = \sum_{i=1}^{\widetilde{k}} (d_i - 1) + \widetilde{\ell}\, ,
$$
where
$$
0 \leq \widetilde{k} < k , \quad
\widetilde{\ell} = \ell - m + \sum_{i=\widetilde{k}+1}^{k} (d_i - 1) > 0
\quad \text{and} \quad
\ell - m + \sum_{i=\widetilde{k}+2}^{k} (d_i - 1) \leq 0\;,  
$$ 
hence $\widetilde{\ell} \leq d_{\widetilde{k}+1} - 1$.
We want to prove that 
$$
(d_{k+1} - m ) \delta_{\mathcal{X}_{\widehat{k+1}}} (d - m) \geq 
\delta_{\mathcal{X}}(d) = (d_{k+1} - \ell) 
\prod_{i=k+2}^{n} d_i,
$$ 
and from $k \geq \widetilde{k} + 1$ we get $k+1 \in \{\widetilde{k}+2,\ldots , 
n \}$, so that  
\[ 
\delta_{\mathcal{X}_{\widehat{k+1}}} (d - m)  =
(d_{\widetilde{k}+1} - \widetilde{\ell}) 
\prod_{i=\widetilde{k}+2 , i \neq k+1}^{n} d_i .
\]
Thus we must verify that
\begin{equation}\label{desigM}
(d_{\widetilde{k}+1} - \widetilde{\ell}) 
\left( \prod_{i=\widetilde{k}+2}^{k} d_i \right)
(d_{k+1} - m ) \geq
(d_{k+1} - \ell) .
\end{equation}
Let $M$ be the function defined by
$$
M(a_{\widetilde{k}+1},\ldots,a_{k+1})
= (d_{\widetilde{k}+1} - a_{\widetilde{k}+1}) .\cdots.
(d_{k+1}-a_{k+1}),
$$
where $a_i$ is a nonnegative integer less than $d_i$, for $i = 	
\widetilde{k}+1, \ldots, k + 1$, and
$a_{\widetilde{k}+1} + \cdots + a_{k+1} \leq \widetilde{\ell}+m$.
We have studied this function in \cite{carvalho} and \cite{car-neu2017}.   
From $\widetilde{\ell}+m =
\sum_{i=\widetilde{k}+1}^{k} (d_i - 1) + \ell$ and 
\cite[Lemma 2.1]{carvalho} we get $d_{k+1} - \ell$ is the minimum of $M$ so 
that inequality \eqref{desigM} holds. To find out when \eqref{desigM}
is an equality we will use 
results from \cite{car-neu2017}, and for that we define a tuple 
$(a_{\widetilde{k}+1} , \ldots , a_{k+1})$ to be 
normalized if whenever $d_{i-1} < d_i = \cdots = d_{i+s}<d_{i+s+1}$
we have $a_i \ge a_{i+1} \ge \cdots \ge a_{i+s}$.
From \cite[Lemma 2.2]{car-neu2017} we get that the normalized tuples which 
reach the minimum of $M$ are exactly of the type:
\begin{enumerate}
\item  
$(a_{\widetilde{k}+1},\ldots,a_{k+1}) = 
(d_{\widetilde{k}+1}  -1 , \ldots d_k - 1 , \ell)$, or 
\item
$(a_{\widetilde{k}+1},\ldots,a_{k+1}) = 
(d_{\widetilde{k}+1}  -1 , \ldots , d_j - (d_{k+1}-\ell),\ldots,
d_{k+1} - 1)$. 
\end{enumerate}
Type 2 is only possible if 
$d_{k+1}-\ell \leq d_j <d_{k+1}$, we also note that 
if $\ell = d_{k+1}-1$ then types 1 and 2 are the same so
we also assume in type 2 that $\ell < d_{k+1}-1$.
Thus we have equality in \eqref{desigM}  if and only if the tuple 
$(\widetilde{\ell}, 
0, \dots, 0, m)$, when normalized, is equal to $(d_{\widetilde{k}+1}  -1 , 
\ldots d_k - 1 , \ell)$ or 
$(d_{\widetilde{k}+1}  -1 , \ldots , d_j - (d_{k+1}-\ell),\ldots,
d_{k+1} - 1)$.

In the first case, since we don't have any zero entries in 
$(d_{\widetilde{k}+1}  -1 , 
\ldots d_k - 1 , \ell)$ we must have 
$\widetilde{k}+1=k$ and the tuple $(\widetilde{\ell},m)$ when normalized is 
equal to $(d_k - 1 , \ell)$,  
thus we must have either $(\widetilde{\ell},m) = 
(d_k - 1 , \ell)$ or $(m, \widetilde{\ell}) = (d_k - 1 , \ell)$. 
If $(\widetilde{\ell},m) = 
(d_k - 1 , \ell)$ then $m = \ell$, and if $(m, \widetilde{\ell}) = (d_k - 1 , 
\ell)$, then $m = d_k - 1$ and 
from the 
definition of normalized tuple we also must 
have $d_k = d_{k + 1}$.
On the other 
hand if $m = \ell$, from 
$
d  = \sum_{i=1}^{k} (d_i - 1) + \ell
$
we get 
$$
d - m = \sum_{i=1}^{k-1} (d_i - 1) + (d_{k} - 1)
$$
so we must have $\widetilde{k} = k - 1$ and  $\widetilde{\ell} = d_k - 1$, 
hence $(\widetilde{\ell},m) = 
(d_k - 1 , \ell)$.
And if $m = d_k - 1 = d_{k + 1} - 1$, from 
$
d  = \sum_{i=1}^{k} (d_i - 1) + \ell
$
we get 
$$
d - m = \sum_{i=1}^{k-1} (d_i - 1) + \ell
$$
so we must have $\widetilde{k} = k - 1$ and  $\widetilde{\ell} = \ell$, 
hence $(m, \widetilde{\ell}) = 
(d_k - 1 , \ell)$.

The upshot of this is that 
$(\widetilde{\ell},m)$ when normalized is 
equal to $(d_k - 1 , \ell)$
if and only if 
$m  = \ell$ 
or both $m = d_{k + 1} - 1$
and $d_k = d_{k + 1}$.

In the second case, since we may have at most only one zero entry in 
$$
(d_{\widetilde{k}+1}  -1 , \ldots , d_j - (d_{k+1}-\ell),\ldots,
d_{k+1} - 1),	
$$
we must have
$\widetilde{k} + 1=k$
or $\widetilde{k} + 2=k$.
If $\widetilde{k} + 1=k$
then the above tuple is an ordered pair, 
and since it is a type 2 tuple we must 
have that $d_k < d_{k + 1}$ and that this pair is  $(d_k - 
(d_{k+1}-\ell),d_{k+1} - 1)$. Since $d_k < d_{k + 1}$ the tuple 
$(\widetilde{\ell},m)$ is already 
normalized, and if  $(\widetilde{\ell},m) = (d_k - (d_{k+1}-\ell),d_{k+1} - 1)$ 
then $m = d_{k+1} - 1$ and $\widetilde{\ell} = d_k - (d_{k+1}-\ell)$ so that 
$d_k - (d_{k+1}-\ell) > 0$. On the other 
hand if $m = d_{k+1} - 1$ and $d_k - (d_{k+1}-\ell) > 0$, from 
$
d  = \sum_{i=1}^{k} (d_i - 1) + \ell
$
we get 
$$
d - m = \sum_{i=1}^{k} (d_i - 1) + \ell - (d_{k + 1} - 1) =
\sum_{i=1}^{k - 1} (d_i - 1) + d_k - (d_{k + 1} - \ell) 
$$
so we must have $\widetilde{k} = k - 1$ and  $\widetilde{\ell} = d_k - 
(d_{k+1}-\ell)$, hence $(\widetilde{\ell},m) = (d_k - (d_{k+1}-\ell),d_{k+1} - 
1)$.

If $\widetilde{k} + 2 = k$ then we must have $d_k < d_{k + 1}$ so the 
tuple $(\widetilde{\ell},0,m)$ is already  normalized, and if 
$(\widetilde{\ell},0,m) = (d_{k-1}-1,d_k - (d_{k+1}-\ell),d_{k+1} - 1)$ then 
$d_k = d_{k+1}-\ell$ and $m=d_{k+1}-1$. On the other hand if  
$m = d_{k+1} - 1$ and $d_k - (d_{k+1}-\ell) = 0$ from 
$
d  = \sum_{i=1}^{k} (d_i - 1) + \ell
$
we get 
$$
d - m = \sum_{i=1}^{k} (d_i - 1) + \ell - (d_{k + 1} - 1) =
\sum_{i=1}^{k - 2} (d_i - 1) + d_{k - 1} - 1 
$$
so we must have $\widetilde{k} = k - 2$ and  $\widetilde{\ell} = d_{k - 1}
- 1$, hence $(\widetilde{\ell},0,m) = (d_{k-1}-1,d_k - (d_{k+1}-\ell),d_{k+1} - 
1)$.

Thus we have equality in \eqref{desigM}  if and only if
$m=\ell$ or both $m=d_{k+1}-1$ and $d_k \geq d_{k + 1} - \ell$.
\end{proof}
%
%
	
\begin{proposition}\label{minimal}
Let $f$ be a nonzero function in $C_\mathcal{X}(d)$,  and let $h \in C_{\X}(d)$ 
be such that $h = x_{j} \circ \varphi$, where 
$\varphi \in \aff(\X)$
and $1 \leq j \leq k+1$. 
Let  $m > 0$ be the number of $\alpha \in K_{j}$ 
such that 
$Z_\mathcal{X}(h - \alpha) \subset Z_\mathcal{X}(f)$.
Let $g = f \circ \varphi^{-1}$,  
then 
$|f| = \delta_{\mathcal{X}} (d)$ if and only if 
$|g_{\alpha}^{(j)}| =
\delta_{\mathcal{X}_{\widehat{j}}} (d - m)$
whenever $g_{\alpha}^{(j)} \neq 0$, with
$\alpha \in K_{j}$ 
and $m$ satisfies one of the following:\\
1) If $1 \leq j \leq k$ then
$m=d_j - 1$ or 
both $m=\ell + d_j  - d_{k+1}$ and $d_j > d_{k+1} - \ell$. \\
2) If $j=k+1$ then
$m=\ell$ or both $m=d_{k+1}-1$ and $d_k \geq d_{k + 1} - \ell$.
\end{proposition}

%
%
%
\begin{proof}
Let $j \in \{ 1, \ldots, k + 1 \}$.
As in the beginning of the proof of Lemma \ref{desig_dist} we may assume that 
$\varphi$ is the identity, so that $h = x_{\textcolor{magenta}{j}}$.  From the proof of Lemma 
\ref{desig_dist} we get 
$$
|f| = \sum_{\alpha \in K_{\textcolor{magenta}{j}}} |f_\alpha^{(\textcolor{magenta}{j})}|
\ge (d_{\textcolor{magenta}{j}} - m )
 \delta_{\mathcal{X}_{\widehat{\textcolor{magenta}{j}}}} (d - m)
$$ 
and equality holds if and only if 
$|f_{\alpha}^{(\textcolor{magenta}{j})}| 
= \delta_{\mathcal{X}_{\widehat{\textcolor{magenta}{j}}}} (d - m)$
whenever $f_{\alpha}^{(\textcolor{magenta}{j})} \neq 0$, with
$\alpha \in K_{j}$. 
From the two previous Lemmas we know that  
$ \delta_{\mathcal{X}_{\widehat{j}}} (d - m) \geq \delta_{\mathcal{X}} (d)$ and 
we also know when equality holds. 
\end{proof}
	

%
%

As mentioned in the paragraph preceding Lemma \ref{desigparcial} we 
are investigating 
when $(d_j - m ) \delta_{\mathcal{X}_{\widehat{j}}} (d - m) \geq 
\delta_{\mathcal{X}} (d)$ holds, and under which conditions equality holds.
Now we treat the case where $m = 0$.

\begin{lemma}\label{desigparcial0}
Let $1 \leq j \leq k+1$. We have
$$
d_{j} \delta_{\mathcal{X}_{\widehat{j}}}
(d) \geq \delta_{\mathcal{X}} (d)
$$
with equality if and only if $d_j = d_{k+1}-\ell$ or $d_j = d_{k+2}$.
\end{lemma}
\begin{proof}
If $d_j \leq d_{k+1} - \ell$ we may write
$$
d  = \sum_{i=1,i \neq j}^{k} (d_i - 1) + \ell + (d_j - 1),
$$
so that  $\delta_{\mathcal{X}_{\widehat{j}}} (d ) = 
(d_{k+1} - (\ell + d_j - 1)) 
\displaystyle
\prod_{i=k+2}^{n} d_i$.
From 
$$
d_j (d_{k+1} - (\ell + d_j - 1)) - (d_{k+1} - \ell))   =
(d_j - 1)(d_{k+1} - \ell -d_j ) \geq 0
$$
we get 
$$
d_{j} \delta_{\mathcal{X}_{\widehat{j}}}
(d) \geq  (d_{k+1} - \ell) 
\prod_{i=k+2}^{n} d_i = \delta_{\mathcal{X}} (d),
$$
with equality if and only if $d_j = d_{k+1} - \ell$.

If $d_j > d_{k+1} - \ell$ we may write
$$
d = \sum_{i=1,i \neq j}^{k+1} (d_i - 1) + \ell + d_j - d_{k+1} \, ,
$$
so that  $\delta_{\mathcal{X}_{\widehat{j}}} (d ) = 
(d_{k+2} - (\ell + d_j - d_{k+1})) 
\displaystyle
\prod_{i=k+3}^{n} d_i$.
From 
$$
d_j	(d_{k+2} - (\ell + d_j - d_{k+1})) - (d_{k+1} - \ell)d_{k+2}  =
(d_j - (d_{k+1} - \ell)) 
(d_{k + 2} - d_j) \geq 0
$$
we get 
$$
d_{j} \delta_{\mathcal{X}_{\widehat{j}}} (d) \geq
\delta_{\mathcal{X}} (d),
$$
with equality if and only if $d_j = d_{k+2}$.
\end{proof}

%
%
\begin{proposition} \label{d1dk}
	Let $f \in C_{\mathcal{X}}(d)$ and suppose that
	$d_j < d_{k+1}- \ell$
	for some $1 \leq j \leq k$. If $|f| = \delta_\mathcal{X} (d)$ then
	the number of $\alpha \in K_{j}$ 
	such that 
	$Z_\mathcal{X}(x_j - \alpha) \subset Z_\mathcal{X}(f)$
	is $d_j - 1$ and for $\alpha \in K_{j}$ such that
	$f_\alpha^{(j)} \neq 0$ we have
	$|f_\alpha^{(j)}| = |f| = \delta_\mathcal{X} (d)
	= \delta_{\mathcal{X}_{\widehat{j}}} (d - (d_j - 1))$.
\end{proposition}
%
%
\begin{proof}
Let $m$ be the number of $\alpha \in K_j$ such that 
$Z_\mathcal{X}(x_j - \alpha) \subset Z_\mathcal{X}(f)$.
By Lemma \ref{desig_dist} we have 
$|f| \ge (d_j - m ) \delta_{\mathcal{X}_{\widehat{j}}} (d - m)$.
As $d_j < d_{k+1}- \ell$ and $|f| = \delta_{\mathcal{X}} (d)$,
from Lemma \ref{desigparcial0} we get $m>0$ and
from Lemma \ref{numericoj} we have
$m=d_j - 1$ and
$\delta_{\mathcal{X}_{\widehat{j}}} (d - (d_j - 1))
=\delta_\mathcal{X} (d)$.
We conclude by observing that
for the only element $\alpha \in K_j$ such that
$f_\alpha^{(j)} \neq 0$ we have
$|f| =
|f_\alpha^{(j)}|$.
\end{proof}
%
%

\section{Main results}

As in the preceding section we continue to write $d$ as in the statement of 
Theorem \ref{min-dist}, namely $d = \sum_{i =1}^k (d_i -1) + 
\ell
$,  with $0 < \ell \leq d_{k + 1} - 1$ (and $0 \leq k < n$).
The next result describes the minimal weight codewords of affine cartesian 
codes for the lowest range of values of $d$, meaning the case when $k = 0$.

%
%
\begin{proposition}\label{minimalcodewords1}
Let $1 \le d < d_1$, the minimal weight codewords of $C_\mathcal{X}(d)$
are $\mathcal{X}$-equivalent to the functions 
$$
\sigma \prod_{i=1}^{\ell} (x_1 - \alpha_i)\, ,
$$
with $\sigma \in \K^*$, $\alpha_i \in K_1$ and 
$\alpha_i \neq \alpha_j$ for $1 \le i \neq j \le \ell$.
\end{proposition}
%
%
\begin{proof}
Let $f \in C_\mathcal{X}(d)$ be such that $| f | = \delta_{\mathcal{X}} (d)$.
From  Corollary \ref{fatores} we get that  $f$ has a degree one factor
$h$ which is  $\mathcal{X}$-equivalent to $x_1$.
Let $m \le d = \ell$ be the  number of distinct elements $\alpha \in K_1$ such 
that 
$Z_\mathcal{X}(x_1 - \alpha) \subset Z_\mathcal{X}(f)$.

As $m \leq d$, from
Proposition \ref{minimal} (2) we have	
$|f| = \delta_{\mathcal{X}} (d)$ if and only if
$m=\ell$.
Now the result follows from Lemma \ref{divide2}.
\end{proof}

Now we
describe the minimal weight codewords  for the case where
$\ell = d_{k + 1} - 1$ and $0 \leq k < n$.

%
%
\begin{proposition}\label{pesominimo}
The minimal weight codewords of $C_\mathcal{X}(d)$, for 
	$d = \displaystyle \sum_{i=1}^{k+1} (d_i - 1)$, $0 \leq k < n$,
	are $\mathcal{X}$-equivalent to the functions of the form 
	$$
	\sigma \prod_{i=1}^{k+1} (1 - x_i^{d_i-1})\, ,
	$$
	with $\sigma \in \K^{*}$.	
\end{proposition}
%
%
\begin{proof}
We will prove the result by induction on $k$, and we note that 
the case $k = 0$ is already covered  by Proposition \ref{minimalcodewords1}, so
 we assume $k > 0$ and that the result holds for $k - 1$. 

Let $f \in C_\mathcal{X}(d)$ be such that $| f | = \delta_{\mathcal{X}} (d)$.
From  Corollary \ref{fatores} we get that  $f$ has a degree one factor
$h$ such that  
$h = x_{k+1} \circ \varphi$, for some
$\varphi \in \aff(\X)$.
Let  $m > 0$ be the number of $\alpha \in K_{k + 1}$ 
such that 
$Z_\mathcal{X}(h - \alpha) \subset Z_\mathcal{X}(f)$.
From Proposition  \ref{minimal} (2) we get 
$m=d_{k+1}-1$ (since $\ell = d_{k+1}$).
In particular $f_{\alpha}^{(k+1)}\neq 0$ for only one value of 
$\alpha \in K_{k+1}$, and without loss of generality, we may assume
that $\varphi$ is the  identity transformation and  $\alpha=0$.
Hence, from  Lemma \ref{divide2} we get 
$$
f = (1 - x_{k+1}^{d_{k+1} - 1}) g\, ,
$$
for some $g \in C_\mathcal{X}(d - (d_{k+1} - 1))$.
Let $P$ and $Q$ be the reduced polynomials associated to  $f$ and $g$, 
respectively. Then 
$$
P - (1 - X_{k+1}^{d_{k+1} - 1})Q
$$
is in the ideal $I_\X = (X_1^{d_1} - X_1, \ldots, X_n^{d_n} - X_n )$. Write 
$Q = Q_1 + X_{k+1}Q_2$, where $Q_1$ and $Q_2$ are reduced polynomials 
and    
$X_{k+1}$ does not appear in any monomial of 
 $Q_1$. Then $P - (1 - X_{k+1}^{d_{k+1} - 1})Q_1$
is in $I_\X$, and 
writing $g_1$ for the function associated to $Q_1$, we get
$f = (1 - x_{k+1}^{d_{k+1} - 1}) g_1$. Since 
$\deg(Q_1) = d - (d_{k+1}-1)$ we  have
 $g_1 \in C_{\mathcal{X}_{\widehat{k+1}}} 
(d - (d_{k+1} - 1))$, and from 
$d - (d_{k+1} - 1) = \displaystyle \sum_{i=1}^{k} (d_i - 1)$,   
$\delta_\mathcal{X} (d)
	= \delta_{\mathcal{X}_{\widehat{k + 1}}} (d - (d_{k+1} - 1))$
and $| f | = | g_1 |$ we see that $g_1$ is a minimal weight codeword of 	
$C_{\mathcal{X}_{\widehat{k+1}}} 
(d - (d_{k+1} - 1))$
so we may apply 
the induction hypothesis to $g_1$, which concludes the proof of the 
Proposition.
%
%
\end{proof}

%
%

%
%
\begin{lemma}\label{soma1}
Let $d=\displaystyle \sum_{i=1}^{k+1} (d_i - 1)$, $0 \leq k < n$
and let $g \in C_{\mathcal{X}}(d)$ be
	such that $|g|=\delta_{\mathcal{X}}(d)$. Let 
	$h \in C_{\mathcal{X}}(d-s)$, where $0 < s \le d_1 - 1$.
	If $f=g+h$ then $|f| \ge (s+1) \delta_{\mathcal{X}}(d)$ or
	$|f| = s\delta_{\mathcal{X}}(d)$.
	From the above Proposition there exists
	$\varphi \in \aff(\mathcal{X})$
	such that $g \circ \varphi^{-1} = 
	\sigma \prod_{i=1}^{k+1} (1 - x_i^{d_i-1})\, ,
	$
with $\sigma \in \K^{*}$.
Let $\widehat{f} = f\circ \varphi^{-1}$,
	 if
	$|f|=s \delta_{\mathcal{X}}(d)$
	then, for each $1 \le j \le k+1$, 
	the number of elements $\alpha \in K_j$
	such that
	$Z_\mathcal{X}(x_j - \alpha) \subset Z_\mathcal{X}(\widehat{f})$
	is either 
	$d_j - 1$ or $d_j - s$.
\end{lemma}
%
%
\begin{proof}
As in the proof of Lemma \ref{desig_dist} we may assume that $\varphi$ is the 
identity transformation, so we identify $\widehat{f}$ with $f$ and $g \circ 
\varphi^{-1}$ with $g$.

We will make an induction on $n$. 
	If $n=1$ then  $k=0$,  $d =d_1 - 1$, $j=1$ and $|g|=1$. Since
	$h \in C_{\mathcal{X}}(d_1-(s+1))$ and $ |K_1| = d_1$ we have  
	 $|h| \ge s+1$, and a fortiori 
	 $|f|\ge s$. 
	 If $|f|=s$ then there are $d_1 - s$ elements $\alpha \in K_1$ such that
	 $Z_\mathcal{X}(x_1 - \alpha) \subset Z_\mathcal{X}(f)$.
	
	We will do an induction on $n$, so we assume that the result is true for $n-1$ and let $j \in \{1, \ldots, k+1\}$.
From the hypothesis on $g$ 
and using the notation established in Definition \ref{f_alpha} 
we may write 
	$g=(1 - x_j^{d_j-1})g^{(j)}_0$, where
	$g^{(j)}_0 \in C_{\mathcal{X}_{\widehat{j}}}(d-(d_j - 1))$ is a function of  minimal weight.
We also write 
	$$
	|f| = \sum_{\alpha \in K_j} |f^{(j)}_\alpha| =
	|g^{(j)}_0 + h^{(j)}_0| + \sum_{\alpha \in K_j^*} |h^{(j)}_\alpha|.
	$$
	
	Let's assume that $h^{(j)}_0=0$, since  
	$\delta_{\mathcal{X}_{\widehat{j}}}(d - (d_j-1)) = \prod_{i=k+2}^{n} d_i = \delta_{\mathcal{X}}(d) $  and 
	$  \delta_{\mathcal{X}}(d - s) = (d_{k+1} - (d_{k+1} - 1 - s)) \prod_{i=k+2}^{n} d_i  $ 
	we get
	$$
	|f|
	= |g^{(j)}_0 | + |h| \ge
	\delta_{\mathcal{X}_{\widehat{j}}}(d - (d_j-1)) + \delta_{\mathcal{X}}(d - s)
	=
	(s+2)\delta_{\mathcal{X}}(d)
	$$
	which proves the Lemma in this case.
	Assume now that $h^{(j)}_0 \neq 0$, and let 
	$m$ 
	be the number of elements  $\alpha \in K_j$
	such that
	$Z_\mathcal{X}(x_j - \alpha) \subset Z_\mathcal{X}(h)$.
	
	Let's assume that  
	$f^{(j)}_0 \neq 0$, in this case  
	$m$ is also the number of elements  $\alpha \in K_j$
	such that 
	$Z_\mathcal{X}(x_j - \alpha) \subset Z_\mathcal{X}(f)$
	since $g=(1 - x_j^{d_j-1})g^{(j)}_0$.
	If $m=d_j-1$ then from Lemma \ref{divide2} we have $h=(1 - 
	x_j^{d_j-1})\widetilde{h}$, with $\widetilde{h} \in 
	C_{\mathcal{X}}(d-(d_j - 1) -s)$.
	As in the end of the proof of Proposition \ref{pesominimo} we may assume 
	that $\widetilde{h} \in C_{\mathcal{X}_{\widehat{j}} } 
	(d - (d_{j} - 1) - s) $  so that 
	$\widetilde{h} =  h^{(j)}_0$. We now apply the induction hypothesis 
	to $f^{(j)}_0 =  g^{(j)}_0 + h^{(j)}_0$ and we get 
	\[
	| f^{(j)}_0  | \geq (s + 1) \delta_{\mathcal{X}_{\widehat{j}}}(d - 
	(d_j-1)) 
	\textrm{ or }
	| f^{(j)}_0  | = s  \delta_{\mathcal{X}_{\widehat{j}}}(d - (d_j-1)).
	\]
If $| f^{(j)}_0  | = s  \delta_{\mathcal{X}_{\widehat{j}}}(d - (d_j-1))$ then, 
from the induction hypothesis, we get that for $i \neq j$ there are $d_i - 1$ 
or $d_i - s$ values of $\alpha \in K_i$ such that  	
$Z_{\mathcal{X}_{\widehat{j}}}(x_j - \alpha) \subset 
Z_{\mathcal{X}_{\widehat{j}}}(f^{(j)}_0)$
and from $f = g + h = (1 - x_j^{d_j-1})g^{(j)}_0 + (1 - x_j^{d_j-1})h^{(j)}_0 = 
(1 - x_j^{d_j-1})f^{(j)}_0$ we get the statement of the Lemma for the case 
where $h^{(j)}_0 \neq 0$, $f^{(j)}_0 \neq 0$ and $m=d_j-1$. Still assuming that
$h^{(j)}_0 \neq 0$ and $f^{(j)}_0 \neq 0$, we now treat the case where   
$0 \le m < d_j -1$. From Lemma \ref{divide2} we know that 
$h = \prod_{i=1}^{m} (x_j - \alpha_i) \widetilde{h}$, where $\alpha_1, \ldots, 
\alpha_m \in K_j^*$ and $\widetilde{h} \in 
C_{\mathcal{X}}(d - s - m)$ so 	$h^{(j)}_0 = \beta \widetilde{h}^{(j)}_0$, 
with $\beta \in K_j^*$ and we get $h^{(j)}_0 \in  
C_{\mathcal{X}_{\widehat{j}}}(d - s - m)$ (note that we also get 
$h^{(j)}_\alpha \in  
C_{\mathcal{X}_{\widehat{j}}}(d - s - m)$ for all $\alpha \in K^*_j \setminus 
\{\alpha_1,\ldots, \alpha_m\}$). Thus, from 
$f^{(j)}_0 =  g^{(j)}_0 + 
h^{(j)}_0$ we get that the degree of $f^{(j)}_0$ is at most   
$\max \{ d - (d_j - 1) , d - (s+m)\}$. We now consider the following cases.	
%
	\begin{enumerate}
		\item Assume that $d_j - 1 < s + m$, so we have that the degree of   
		$ f^{(j)}_0 = g^{(j)}_0 + h^{(j)}_0$ is at most $d - (d_j - 1)$. 
		From $h^{(j)}_0 \in  
		C_{\mathcal{X}_{\widehat{j}}}(d - (s + m))$ and writing 
		$d - (s + m) = d - (d_j - 1) - (s + m - (d_j - 1))$, we observe that
		$0 < s + m - (d_j - 1)  = s - (d_j - 1 - m ) < d_1 - 1$, so we may 
		apply the induction hypothesis on $f^{(j)}_0$ and we get, in 
		particular, that $|f^{(j)}_0| \geq (s+m - (d_j - 1)) 
		\delta_{\mathcal{X}_{\widehat{j}}}(d - (d_j - 1))
				=
				(s+m +1 - d_j) \delta_{\mathcal{X}}(d)$.
From $
	|f| = 
	|f^{(j)}_0| + \sum_{\alpha \in K_j^*} |h^{(j)}_\alpha|
	$
and the fact that $|h^{(j)}_\alpha| \geq \delta_{\mathcal{X}_{\widehat{j}}}(d - 
(s + m))$  for all $\alpha \in K^*_j \setminus 
\{\alpha_1,\ldots, \alpha_m\}$ we get 
		$|f|  \ge  (s+m +1 - d_j) \delta_{\mathcal{X}}(d) + 
		(d_j - m - 1)\delta_{\mathcal{X}_{\widehat{j}}}(d - (s+m))$. 
We claim that 
		$$
		\delta_{\mathcal{X}_{\widehat{j}}}(d - (s+m)) =
		(s + m - d_j + 2)
		\delta_{\mathcal{X}}(d),
		$$
and to prove this fact we have to consider the cases where $j \leq k$  and 
$j = k + 1$. We will do the case $j \leq k$ since the proof of the other case is 
similar to this one.		
So let $j \leq k$, then 
\[
d - (s + m) = \sum_{i = 1, i \neq j}^k  (d_i - 1) + (d_{k + 1} - 1 + d_j  - 1 - 
(s + m))
\]
so 
\[
\delta_{\mathcal{X}_{\widehat{j}}}(d - (s+m)) = 
(d_{k + 1} - ( d_{k + 1} - 1 + d_j  - 1 - 
(s + m)  )\prod_{i = k + 2}^n d_i = (s + m - d_j + 2) \delta_{\mathcal{X}}(d).
\]
Thus 
	\begin{eqnarray}
		|f| & \ge &  (s+m +1 - d_j) \delta_{\mathcal{X}}(d)
		+ (d_j - m - 1) (s + m - d_j + 2) \delta_{\mathcal{X}}(d)
		\nonumber \\
		&  =  & \left( 
		s + (d_j - m - 1) (s + m + 1 - d_j) 
		\right)  \delta_{\mathcal{X}}(d)
		\ge   ( s + 1 ) \delta_{\mathcal{X}}(d)
		\nonumber
		\end{eqnarray}
which proves the Lemma in this case.

		\item Assume  now that $d_j - 1 \ge s + m$, in this case 
		$\deg (g^{(j)}_0 + h^{(j)}_0) \le d - (s + m)$, and we have
		\begin{eqnarray}
		|f| & \ge & 
		(d_j - m)\delta_{\mathcal{X}_{\widehat{j}}}(d - (s+m)) =
		(d_j - m ) (d_{k+1} + s + m + 1 -d_j)
		\prod_{i=k+2}^n d_i
		\nonumber\\
		&  =  & \left( 
		(s+1)d_{k+1}+ (d_{k+1} + m - d_j) (d_j- 1 - s - m) 
		\right)  \prod_{i=k+2}^n d_i \ge (s+1) \delta_{\mathcal{X}}(d) \, .
		\nonumber
		\end{eqnarray}
	\end{enumerate}
	We now consider the case  $f^{(j)}_0 = g^{(j)}_0 + h^{(j)}_0 = 0$, so in particular 
	$\deg h^{(j)}_0 = \deg g^{(j)}_0 = d - (d_j -1)$. 
	On the other hand $\deg h_0^{(j)} \leq d - (s + m)$, so we get $s + m \leq d_j - 1$.
	Let $\lambda = ((-1)^m\prod_{i=1}^{m} \alpha_i)^{-1}$, then, using Lemma \ref{divide2} we get
	that there exists a function $\widehat{h}$ such that 
	\[
	h - \lambda \left( \prod_{i=1}^{m} (x_j - \alpha_i) \right) h_0^{(j)} = \lambda x_j \left( \prod_{i=1}^{m} (x_j - \alpha_i) \right) \widehat{h} 	
	\]
Observe that  $\deg (h - \lambda ( \prod_{i=1}^{m} (x_j - \alpha_i) ) h_0^{(j)} ) \leq d - s$ hence 
$\deg \widehat{h} \leq d - (s + m +1)$.
%
%
	
	Assume that $s+m < d_j - 1$ hence $s+m +1 \le d_j - 1$.
	From 
	\[
	h = \lambda \left( \prod_{i=1}^{m} (x_j - \alpha_i) \right) ( h_0^{(j)} +  x_j \widehat{h}) 
	\]
Recall that 
	$$
	|f| =
	 \sum_{\alpha \in K_j^*} |h^{(j)}_\alpha|
	$$
and $h^{(j)}_\alpha = \lambda \left( \prod_{i=1}^{m} (\alpha - \alpha_i) \right) ( h_0^{(j)} +  \alpha \widehat{h}^{(j)}_\alpha) \neq 0$ for $d_j - (m + 1)$ values of $\alpha \in K_j^{*}$. Observe that  
 $\deg h^{(j)}_0 =  d - (d_j - 1) \le d - (s+m +1)$ and
 since $\deg \widehat{h} \le d - (s+m +1)$ we get 
 $\deg h_\alpha^{(j)}\le d - (s + m +1) $, when $h_\alpha^{(j)}\neq 0$. Then  
	\begin{eqnarray*}
	|f| & \ge &  
	(d_j - (m + 1))\delta_{\mathcal{X}_{\widehat{j}}}(d - (s + m + 1)) \\
	 &	= &  (d_j - (m + 1)) (d_{k+2} - (d_j - (s + m +2)))
	\prod_{i=k+3}^n d_i
	\\
	&  =  & \left( 
	(s+1)d_{k+2}+ (d_j - (s + m + 2)) (d_{k+2} - d_j + m + 1)
	\right)  \prod_{i=k+3}^n d_i \ge (s+1) \delta_{\mathcal{X}}(d) .
	\nonumber
	\end{eqnarray*}
	If  $s+m = d_j - 1$ then $\deg h_\alpha^{(j)} = d - (d_j -1)$ whenever
	 $\alpha \in K_j^*$ and $h_\alpha^{(j)}\neq 0$. In this case
	$$
	|f| \ge
	(d_j - m - 1)\delta_{\mathcal{X}_{\widehat{j}}}(d - (d_j -1)) = s  \delta_{\mathcal{X}}(d),
	$$
	and equality holds if and only if $|f_\alpha^{(j)}|=|h_\alpha^{(j)}| =
	\delta_{\mathcal{X}_{\widehat{j}}}(d - (d_j -1))$,
	for all $f_\alpha^{(j)}\neq 0$. Observe that in this case the number of elements
	 $\alpha \in K_j$
	such that 
	$Z_\mathcal{X}(x_j - \alpha) \subset Z_\mathcal{X}(f)$
	is 
	$m + 1 = d_j - s$.

	Still under the assumption that  $s+m = d_j - 1$ we must prove that 
	if $|f| > s  \delta_{\mathcal{X}}(d)$ then 
	$|f| \ge (s + 1)  \delta_{\mathcal{X}}(d)$.  From the above reasoning we know that if
	$|f| > s  \delta_{\mathcal{X}}(d)$ then there exists 
	$\alpha \in K_j^*$ such that  $h_\alpha^{(j)} \neq 0$ and  
	$|h_\alpha^{(j)}| > \delta_{\mathcal{X}_{\widehat{j}}}(d - (d_j -1))
	= \delta_{\mathcal{X}}(d)$. We recall that 
	\[
	h^{(j)}_\alpha = \lambda \left( \prod_{i=1}^{m} (\alpha - \alpha_i) \right) ( h_0^{(j)} +  \alpha \widehat{h}^{(j)}_\alpha), 
	\] 
	 that  $h_0^{(j)} = - g_0^{(j)}$ is a function, or codeword, of minimal 
	 weight  in
	$C_{\mathcal{X}_{\widehat{j}}}(d - (d_j - 1))$ and that 
	$\deg( \widehat{h}_\alpha^{(j)}) \leq d - (s + m + 1) = d - (d_j - 1) - 1$. 
	From the induction hypothesis, with $s = 1$, we get from
	$|h_\alpha^{(j)}| > \delta_{\mathcal{X}_{\widehat{j}}}(d - (d_j -1))$ 
	that 
	$|h_\alpha^{(j)}| \ge 2 \delta_{\mathcal{X}_{\widehat{j}}}(d - (d_j -1))$.
	Hence, from $
		|f| =
		 \sum_{\alpha \in K_j^*} |h^{(j)}_\alpha|
		$ we get 
	$$
	|f| \ge
	(d_j - m - 2)\delta_{\mathcal{X}_{\widehat{j}}}(d - (d_j -1)) +
	2 \delta_{\mathcal{X}_{\widehat{j}}}(d - (d_j -1)) =
	(s+1) \delta_{\mathcal{X}}(d)\, ,
	$$
	which completes the proof of the Lemma.
\end{proof}
\begin{lemma} \label{poligual}
	Let  $f \in C_{\mathcal{X}}(d)$, where
	$d=\displaystyle \sum_{i=2}^{k+1} (d_i - 1)$, $1 \leq k < n$. If 
	there exist $\alpha_1,\alpha_2 \in K_1$, $\alpha_1 \neq \alpha_2$,
	$|f_{\alpha_1}^{(1)}|=|f_{\alpha_2}^{(1)}| =
	\delta_{\mathcal{X}_{\widehat{1}}}(d)$ then there
	exists $\varphi \in \aff(\mathcal{X})$ such that $x_1 = x_1\circ \varphi$
	and 
	$g_{\alpha_1}^{(1)}= g_{\alpha_2}^{(1)}$,
	where $g=f \circ \varphi$.
\end{lemma}
%
%
\begin{proof}
	From Proposition \ref{pesominimo} we may assume without loss of generality that
	$$
	f_{\alpha_1}^{(1)} = \sigma 
	\prod_{i=2}^{k+1} \left( 1 - x_i^{d_i -1} \right) \, ,
	$$
	with $\sigma \in \mathbb{F}_q^*$.
	Since $f \in C_{\mathcal{X}}(d)$ there exists
	$\hat{f} \in C_{\mathcal{X}}(d-1)$ such that
	$f = f_{\alpha_1}^{(1)}  + (x_1 - \alpha_1) \hat{f}$
	so that
	$f_{\alpha_2}^{(1)} = 
	f_{\alpha_1}^{(1)}  + (\alpha_2 - \alpha_1) \hat{f}_{\alpha_2}^{(1)}$. Since
	$|f_{\alpha_2}^{(1)}| = \delta_{\mathcal{X}_{\widehat{1}}}(d)$,
	we get from Lemma \ref{soma1} (with $s=1$) that
	for each $2 \leq j \leq k + 1$ the number of elements $\alpha \in K_j$ such that 
	$Z_{\mathcal{X}_{\widehat{1}}}(x_j - \alpha) \subset
	Z_{\mathcal{X}_{\widehat{1}}}(f_{\alpha_2}^{(1)})$
	is
	$d_j - 1$. Thus for each $2 \leq j \leq k + 1$ there exists
	$\beta_j \in K_j$ such that $f_{\alpha_2}^{(1)}$ is a multiple of
	$
	\prod_{\alpha \in K_j \setminus \{\beta_j\}} (x_j - \alpha)$. From the equality of
	the reduced polynomials 
	\[
	\prod_{\alpha \in K_j \setminus \{\beta_j\}} (X_j - \alpha) = (X_j - \beta_j)^{d_j - 1} - 1
	\] 	
	we get, by successively applications of Lemma \ref{divide2}, that
	$$
	f_{\alpha_2}^{(1)} = \tau
	\prod_{i=2}^{k+1} \left( 1 - (x_i - \beta_i)^{d_i -1} \right)
	$$
	for some $\tau \in \mathbb{F}_q^*$.
	Observe that from 
	$(\alpha_2 - \alpha_1) \hat{f}_{\alpha_2}^{(1)} = f_{\alpha_2}^{(1)} - 
		f_{\alpha_1}^{(1)}  $
	and 	$\hat{f} \in C_{\mathcal{X}}(d-1)$
	we must have $\tau = \sigma$.   
	If $\beta_j=0$ for all $2 \le j \le k+1$ then 
	$f_{\alpha_1}^{(1)} = f_{\alpha_2}^{(1)}$. Otherwise
		consider a function
	$\varphi \in \aff(\mathcal{X})$ such that $x_1 \circ \varphi=x_1$ and
	$$
	x_j \circ \varphi = x_j + \beta_j \frac{x_1 - \alpha_1}{\alpha_2 - \alpha_1}\, .
	$$
	for all  $2 \le j \le k+1$.
	Let 
	$g = f\circ \varphi$, if $x_1 = \alpha_1$ then
    $x_j \circ \varphi = x_j$, and if  
	$x_1 = \alpha_2$ then 
	$x_j \circ \varphi = x_j  + \beta_j$ for all 
	$2 \le j \le k+1$. Thus
	\[
			g_{\alpha_1}^{(1)} = (f \circ \varphi)_{\alpha_1}^{(1)} = \sigma 
		\prod_{i=2}^{k+1} \left( 1 - ((x_i \circ \varphi)_{\alpha_1}^{(1)} )^{d_i -1} \right) = f_{\alpha_1}^{(1)}  ,	
	\]
	and
	\[
				g_{\alpha_2}^{(1)} = (f \circ \varphi)_{\alpha_2}^{(1)} = \sigma 
			\prod_{i=2}^{k+1} \left( 1 - ((x_i \circ \varphi)_{\alpha_2}^{(1)} - \beta_i)^{d_i -1} \right) = f_{\alpha_1}^{(1)}  ,	
		\]
		hence
	$g_{\alpha_1}^{(1)}=  g_{\alpha_2}^{(1)}$.
\end{proof}

Now we prove the main result of this paper, which generalizes the theorem by 
Delsarte,
Goethals and Mac Williams on minimal weight codewords of $GRM_q(d,n)$ to 
the minimal weight codewords of $C_\mathcal{X}(d)$.

%
%
\begin{theorem}\label{pesominimofinal}
Let 
$d = \displaystyle \sum_{i=1}^{k} (d_i - 1)+ \ell $, $0\le  k < n$ and
$0 < \ell \leq d_{k+1} - 1$, 
the minimal weight codewords of 
 $C_\mathcal{X}(d)$
are $\mathcal{X}$-equivalent to the functions of the  form
$$
g = \sigma \prod_{i=1,i \neq j}^{k+1} (1 - x_i^{d_i-1})
\prod_{t=1}^{d_j - (d_{k+1}-\ell)} 
\left(
x_j - \alpha_t
\right)
\, ,
$$
for some $1 \le j \le k+1$ such that $d_{k+1}- \ell  \le d_j$,
where $0 \neq \sigma \in \K$ and $\alpha_1, \ldots, \alpha_{d_j - (d_{k+1}-\ell)}$
are disticnt elements of $K_j$ (if $d_j = d_{k+1}- \ell$ we take the second product as being equal to 1).
\end{theorem}
%
%
\begin{proof}
If $k=0$ the $d < d_1$ and the result follows from  Proposition \ref{minimalcodewords1}.

We will do an induction on $k$, so let's assume that the result holds for $k - 1$. 

If $\ell = d_{k+1}-1$, then the  result
follows from 
Proposition \ref{pesominimo}.

Let $\ell < d_{k+1}-1$ and let $f \in C_\mathcal{X}(d)$ 
be a minimal weight codeword, i.e.\ $|f| = \delta_{\mathcal{X}}(d)$.
From Corollary \ref{fatores} 
$f$ has a factor which is $\mathcal{X}$-equivalent
to $x_{k+1}$.  Let $1 \leq j \leq k+1$ be least integer such that 
$f$ has a factor which is 
 $\mathcal{X}$-equivalent to $x_j$ and 
and let's assume without loss of generality that $x_j - \alpha$ is a factor of $f$ for some $\alpha \in K_j$. Let 
$m > 0$
be the number of elements of $\alpha \in K_j$
such that
$Z_\mathcal{X}(x_j - \alpha) \subset Z_\mathcal{X}(f)$.
From Proposition \ref{minimal} we get  $m=d_j -1$ or
$m=d_j - (d_{k+1}-\ell)$. 

If $m=d_j -1$ then, after applying an $\X$-affine transformation if necessary, we write 
$$
f = (1 - x_{j}^{d_{j} - 1}) g\, ,
$$
for some $g \in C_\mathcal{X}(d - (d_{j} - 1))$, and as in the
proof of Proposition \ref{pesominimo} we show that actually 
we may write $f$ as 
$$
f = (1 - x_{j}^{d_{j} - 1}) g_1\, ,
$$
with $g_1 \in C_{\mathcal{X}_{\widehat{j}}} 
(d - (d_{j} - 1))$. 
In the case where $1 \leq j \leq k$, 
since $m = d_j - 1$ we get from Lemma \ref{numericoj} that 
$\delta_\mathcal{X} (d)
	= \delta_{\mathcal{X}_{\widehat{j}}} (d - (d_{j} - 1))$
and from $| f | = | g_1 |$ we see that $g_1$ is a minimal weight codeword of 	
$C_{\mathcal{X}_{\widehat{j}}} 
(d - (d_{j} - 1))$,  then 
we may apply the induction hypothesis  to get the result.
In the case where $j = k + 1$, from Proposition \ref{minimal} we also get 
$d_k - (d_{k+1}-\ell) \geq 0$ (besides  $m=d_{k+1} -1$) so from 
Lemma \ref{numericok}  we get  $\delta_\mathcal{X} (d)
	= \delta_{\mathcal{X}_{\widehat{k+1}}} (d - (d_{k+1} - 1))$
and from $| f | = | g_1 |$ we see that $g_1$ is a minimal weight codeword of 	
$C_{\mathcal{X}_{\widehat{k+1}}} 
(d - (d_{k+1} - 1))$.  Writing $d - (d_{k+1} - 1) = \sum_{j = 1}^{k - 1} (d_j - 
1) + (d_k - (d_{k + 1} - \ell))$ we see that, as above, we can apply the 
induction hypothesis to $g_1$, either because $d_k - (d_{k + 1} - \ell) > 0$ or 
because we get the result from Proposition  
\ref{pesominimo} if $d_k = d_{k + 1} - \ell$.

Now we assume that 
$m=d_j - (d_{k+1}-\ell)< d_j -1$. From
Proposition \ref{minimal}
we see that there are $d_{k + 1} - \ell$ elements in 
$K_j$ (say, $\beta_1, \ldots, \beta_{d_{k+1} - \ell}$) such that
for all $i \in 
\{1, \ldots, d_{k+1} - \ell\}$ we get  
$|f_{\beta_i}^{(j)}| = \delta_{\mathcal{X}_{\widehat{j}}}(\widehat{d})$,
with
$$
\widehat{d} = d - (d_j - (d_{k+1}-\ell)) = 
\sum_{i=1, i \neq j}^{k+1} (d_i - 1)\, ,
$$ 
while $|f_{\beta_i}^{(j)}| = 0$ for the other elements of $K_j$ (say, $i 
 \in \{ d_{k+1} - \ell + 1, \ldots, d_j\}$).

From Lemma \ref{divide2} we may write $f$ as  
\begin{equation} \label{fchapeu}
f = \widehat{f} \cdot
\prod_{i=d_{k+1}- \ell + 1}^{d_j} (x_j - \beta_i) 
\end{equation}
with $\widehat{f}\in C_\mathcal{X}(\widehat{d})$.

We treat first the case $j=1$. From Lemma \ref{poligual}, there exists 
$\psi \in \aff(\mathcal{X})$ such that
$x_1 = x_1\circ \psi$,
and $g_{\beta_1}^{(1)}=g_{\beta_2}^{(1)}$, where $g=\widehat{f}\circ \psi$,
and without loss of generality we assume that $\widehat{f}=g$. Observe that
$Z_\mathcal{X}(x_i - \beta_i) \subset 
Z_\mathcal{X}(\widehat{f} - \widehat{f}_{\beta_1}^{(1)}) = Z_\mathcal{X}(
\widehat{f} - \widehat{f}_{\beta_2}^{(1)})$ for $i = 1,2$, so from Lemma 
\ref{divide2} we may write 
$$
\widehat{f} = \widehat{f}_{\beta_1}^{(1)} + (x_1 - \beta_1)(x_1 - \beta_2)h
\, ,
$$
with $h \in C_\mathcal{X}(\widehat{d} - 2)$.
If $d_{k+1}-\ell = 2$, then from 
$\widehat{f}_{\beta_1}^{(1)} = \widehat{f}_{\beta_2}^{(1)}$ and equation 
\eqref{fchapeu}
we may write 
$$
f = \widehat{f}_{\beta_1}^{(1)} \cdot
\prod_{i=3}^{d_1} (x_1 - \beta_i)\, ,
$$
and the result follows from applying Proposition \ref{pesominimo} to 
$\widehat{f}_{\beta_1}^{(1)} \in C_{\mathcal{X}_{\widehat{1}}}(\widehat{d})$. 
If $d_{k+1}-\ell > 2$ then for all 
$2 < t \le d_{k+1}-\ell$ we get 
$$
\widehat{f}_{\beta_t}^{(1)} = \widehat{f}_{\beta_1}^{(1)} +
 (\beta_t - \beta_1)(\beta_t - \beta_2)h_{\beta_t}^{(1)}
\, .
$$
If $h_{\beta_t}^{(1)} \neq 0$  then from Lemma \ref{soma1} (taking $s=2$),
we get $|\widehat{f}_{\beta_t}^{(1)}| \ge 2
\delta_{\mathcal{X}_{\widehat{1}}}(\widehat{d})$, a contradiction. 
Hence
$\widehat{f}_{\beta_1}^{(1)}= \widehat{f}_{\beta_t}^{(1)}$ for all 
$1 \le t \le d_{k+1}- \ell$, and from equation \eqref{fchapeu} we may write 
$$
f = \widehat{f}_{\beta_1}^{(1)} \cdot
\prod_{i=d_{k+1}- \ell + 1}^{d_1} (x_1 - \beta_i)\, ,
$$
with $\widehat{f}_{\beta_1}^{(1)} \in C_{\mathcal{X}_{\widehat{1}}}(\widehat{d})$.
Again, the result follows from applying Proposition \ref{pesominimo} to 
$\widehat{f}_{\beta_1}^{(1)}$, which concludes the case $j = 1$.

Assume now that $j > 1$ and 
let $\mathcal{X}_{\widehat{1, j}} := K_{2} \times \cdots \times K_{j-1}  \times 
K_{j+1} 
\times \cdots  \times K_n$. Then for all  $\alpha \in K_1$ we get 
$Z_\mathcal{X}(x_1 - \alpha) \not \subset Z_\mathcal{X}(f)$
and from  
Proposition  \ref{d1dk} we get 
$d_1 \ge d_{k+1} - \ell$.
From equation \eqref{fchapeu} we get  
$|f_{\beta_t}^{(j)}|=|\widehat{f}_{\beta_t}^{(j)}|$,
so Proposition \ref{minimal} implies 
$| \widehat{f}_{\beta_t}^{(j)} | = \delta_{\mathcal{X}_{\widehat{j}}} 
(\widehat{d})$ for all $t = 1, \ldots, d_{k+1} - \ell$. 
Thus, in particular,  
$\widehat{f}_{\beta_1}^{(j)}$ is $\mathcal{X}_{\widehat{j}}$-equivalent
to a function of the form $(1 - x_1^{d_1 -1}) g_1$, where
$g_1\in C_{\mathcal{X}_{\widehat{1,j}}} \left( 
\displaystyle \sum_{i=2, i \neq j}^{k+1} (d_i -1)\right)$,
and $| g_1 | = \delta_{\mathcal{X}_{\widehat{1,j}}}(\sum_{i=2, i \neq j}^{k+1} (d_i -1))$, so we may assume 
\begin{equation}\label{fhat1}
\widehat{f}_{\beta_1}^{(j)} = (1 - x_1^{d_1 -1}) g_1 \, .
\end{equation}
Using Lemma \ref{divide} there exists $h \in C_{\mathcal{X}}(\widehat{d} - 1)$ such that 
$\widehat{f} = \widehat{f}_{\beta_1}^{(j)} + (x_j - \beta_1) h$, and evaluating both sides at $\beta_t$, with $t \in \{2, \ldots, d_{k+1} - \ell\}$, we get 
$\widehat{f}_{\beta_t}^{(j)} = \widehat{f}_{\beta_1}^{(j)} + (\beta_j - \beta_1) h_{\beta_t}^{(j)}$.
We now may apply Lemma \ref{soma1} (replacing $f$ by $\widehat{f}_{\beta_t}^{(j)}$,   $g$ by $\widehat{f}_{\beta_1}^{(j)}$, $h$ by 
$(\beta_j - \beta_1) h_{\beta_t}^{(j)}$), and using that  
$| \widehat{f}_{\beta_t}^{(j)} | = \delta_{\mathcal{X}_{\widehat{j}}} 
(\widehat{d})$ we may conclude that there are $d_1 - 1$ elements $\alpha$ in $K_1$ such that 
$Z_{\mathcal{X}_{\widehat{j}}}(x_1 - \alpha)  \subset 
Z_{\mathcal{X}_{\widehat{j}}}(\widehat{f}_{\beta_t}^{(j)})$.

From Lemma  \ref{divide2}, for every  $1 \le t \le d_{k+1}- \ell$, there exists  
$\alpha_t \in K_1$ such that 
$$
\widehat{f}_{\beta_t}^{(j)} = (1 - (x_1- \alpha_t)^{d_1 -1}) g_t \, ,
$$
(here we are using that $((x_1- \alpha_t)^{d_1 -1} - 1)(x_1 - \alpha_t) = x_1^{d_1} - x_1$)
where, as in Proposition  \ref{pesominimo},
$g_t \in C_{\mathcal{X}_{\widehat{1,j}}}$ 
is a minimal weight function of degree 
$\displaystyle \sum_{i=2, i \neq j}^{k+1} (d_i -1)$.
Note that from \eqref{fhat1} we get  $\alpha_1 = 0$. We also note that  if there exists $\alpha \in K_1$,  distinct from $\alpha_t$ for all $t \in 
\{1, \ldots, d_{k + 1} - \ell\}$ then all functions  
$\widehat{f}_{\beta_t}^{(j)}$
vanish  in $x_1 = \alpha$, hence 
$Z_\mathcal{X}(x_1 - \alpha)  \subset Z_\mathcal{X}(\widehat{f}) \subset Z_\mathcal{X}(f)$,
a contradiction with the assumption $j > 1$.
Thus 
for all $\alpha \in K_1$  there exists $1 \le t \le d_{k+1}- \ell$
such that $\alpha = \alpha_t$, hence  
$d_1 \le d_{k+1}- \ell$ and a fortiori $d_1 = d_{k+1}- \ell$. 

For each  $t \in \{1, \ldots,  d_{1}\}$ let 
$$
h_t(x_j) = \prod_{i=1, i \neq t}^{d_1} (x_j - \beta_i) 
$$
and let 
$$ u = 
\sum_{i=1}^{d_1} \left( 
\left(1 - (x_1-\alpha_i)^{d_1 - 1}\right) \cdot g_i \cdot
\frac{h_i(x_j)}{h_i(\beta_i)}
\right)\,  \cdot \, 
\prod_{s=d_1+1}^{d_j} (x_j - \beta_s)\, .
$$
Clearly, for $d_1 < t \leq d_j$, from the definition of $u$ and \eqref{fchapeu} we get  $u_{\beta_t}^{(j)} = 0 = f_{\beta_t}^{(j)}$.  For $ t \in \{1, \ldots, d_1\}$  we get 
$$ 
u_{\beta_t}^{(j)}  = 
\left(1 - (x_1-\alpha_t)^{d_1 - 1}\right)  g_t 
\prod_{s=d_1+1}^{d_j} (\beta_t - \beta_s) = \widehat{f}_{\beta_t}^{(j)} \prod_{s=d_1+1}^{d_j} (\beta_t - \beta_s) =
f_{\beta_t}^{(j)}.  
$$
Thus we conclude that $u = f$. Letting $x_1 =
\alpha_t$,
for all  $1 \le t \le d_1$ we get
$$
f_{\alpha_t}^{(1)} = 
 g_t \cdot \frac{h_t(x_j)}{h_t(\beta_t)}
\cdot
\prod_{s=d_1+1}^{d_j} (x_j - \beta_s) .
$$
Observe that 
$h_t(x_j)
\displaystyle \prod_{s=d_1+1}^{d_j} (x_j - \beta_s)$ does not vanish only when 
$x_j=\beta_t$, so $|f_{\alpha_t}^{(1)}|=|g_t|$. 
From 
$$|g_t| =
\delta_{\mathcal{X}_{\widehat{1,j}}}
\left(\displaystyle \sum_{i=2, i \neq j}^{k+1} (d_i -1)\right)
\, , \quad
\delta_{\mathcal{X}_{\widehat{1}}}
\left(\displaystyle \sum_{i=2}^{k+1} (d_i -1)\right) =
\delta_{\mathcal{X}_{\widehat{1,j}}}
\left(\displaystyle \sum_{i=2, i \neq j}^{k+1} (d_i -1)\right)
$$
and
$$
d = \sum_{i=1}^{k} (d_i -1) + \ell = \sum_{i=2}^{k+1} (d_i -1),
$$
we get
$$
|f_{\alpha_t}^{(1)}| =
\delta_{\mathcal{X}_{\widehat{1}}}
\left(\displaystyle \sum_{i=2}^{k+1} (d_i -1)\right)
=
\delta_{\mathcal{X}_{\widehat{1}}}(d).
$$
Thus we get $f \in C_\mathcal{X}(d)$, where
$d = \displaystyle \sum_{i=2}^{k+1} (d_i -1)$ and
$|f_{\alpha_1}^{(1)}|=|f_{\alpha_2}^{(1)}| = 
\delta_{\mathcal{X}_{\widehat{1}}}(d)$.
From Lemma \ref{poligual}, there exists 
$\theta \in \aff(\mathcal{X})$ such that  $x_1=x_1 \circ \theta$
and $\widetilde{f}_{\alpha_1}^{(1)}=\widetilde{f}_{\alpha_2}^{(1)}$, where $\widetilde{f} = f \circ \theta$,
and without loss of generality we assume that $\widetilde{f}=f$. Observe that
$Z_\mathcal{X}(x_1 - \alpha_i) \subset 
Z_\mathcal{X}(f - f_{\alpha_1}^{(1)}) = Z_\mathcal{X}(
f - f_{\alpha_2}^{(1)})$ for $i = 1,2$, so from Lemma 
\ref{divide2} we may write 
$$
f = f_{\alpha_1}^{(1)} + (x_1 - \alpha_1)(x_1 - \alpha_2)
\overline{f},
$$
with $\overline{f} \in C_\mathcal{X}(d - 2)$.
If $d_1 = 2$, then $f=f_{\alpha_1}^{(1)}$. If $d_1>2$ then for
all $t \in \{3,\ldots , d_1\}$ we get
$f_{\alpha_t}^{(1)} = f_{\alpha_1}^{(1)} +
(\alpha_t - \alpha_1)(\alpha_t - \alpha_2)
\overline{f}_{\alpha_t}^{(1)}$.
If $\overline{f}_{\alpha_t}^{(1)}\neq 0$ then from
Lemma \ref{soma1} (taking $s=2$),
we get $|f_{\alpha_t}^{(1)}| \ge 2
\delta_{\mathcal{X}_{\widehat{1}}}(d)$, a contradiction.
Hence we must have 
$f_{\alpha_t}^{(1)} = f_{\alpha_1}^{(1)}$ for all 
$1 \le t \le d_1$ 
and
the result follows from applying Proposition \ref{pesominimo} to 
$f=f_{\alpha_1}^{(1)} \in C_{\mathcal{X}_{\widehat{1}}}(d)$.
\end{proof}

%
%

\bibliographystyle{plain}

\end{document}